\documentclass[11pt,letterpaper]{amsart}

\usepackage{amsfonts,amsmath,amsthm,amssymb}
\usepackage{shortvrb,graphicx,psfrag}
\usepackage{wrapfig}
\usepackage[colorlinks,pdfpagelabels,pdfstartview = FitH,bookmarksopen= true,bookmarksnumbered = true,linkcolor =  ,plainpages =false,hypertexnames = false,citecolor =  ,pagebackref=false]{hyperref}
\usepackage{relsize}
\usepackage{orcidlink}
\usepackage{xcolor}
\usepackage{appendix}
\usepackage{comment}
\usepackage{ulem}
\usepackage[margin=1.4in]{geometry}
\usepackage{stmaryrd}
\usepackage{float}
\usepackage{mathrsfs}
\SetSymbolFont{stmry}{bold}{U}{stmry}{m}{n}

\allowdisplaybreaks
\newcommand{\N}{\mathbb N}
\renewcommand{\d}{\mathrm{d}}
\newcommand{\dx}{\mathrm{d}x}
\newcommand{\dy}{\mathrm{d}y}

\newcommand{\dt}{\mathrm{d}t}

\renewcommand{\rho}{\varrho}

\newcommand{\noi}{\noindent}
\newcommand{\dsty}{\displaystyle}

\newcommand{\pl}{\partial}

\newcommand{\al}{\alpha}

\newcommand{\be}{\beta}

\newcommand{\gm}{\gamma}
\newcommand{\dl}{\delta}

\newcommand{\lm}{\lambda}

\newcommand{\kp}{\kappa}
\newcommand{\varep}{\varepsilon}

\newcommand{\vp}{\varphi}
\newcommand{\sig}{\sigma}

\newcommand{\om}{\omega}

\newcommand{\z}{\zeta}

\newcommand{\loc}{\operatorname{loc}}

\newcommand{\dist}{\operatorname{dist}}
\newcommand{\dvg}{\operatorname{div}}
\newcommand{\essup}{\operatornamewithlimits{ess\,sup}}
\newcommand{\essinf}{\operatornamewithlimits{ess\,inf}}

\newcommand{\rr}{\mathbb{R}}
\newcommand{\rn}{\rr^N}
\newcommand{\nn}{\mathbb{N}}

\def\Xint#1{\mathchoice
    {\XXint\displaystyle\textstyle{#1}}
    {\XXint\textstyle\scriptstyle{#1}}
    {\XXint\scriptstyle\scriptscriptstyle{#1}}
    {\XXint\scriptscriptstyle\scriptscriptstyle{#1}}
    \!\int}
\def\XXint#1#2#3{\setbox0=\hbox{$#1{#2#3}{\int}$}
    \vcenter{\hbox{$#2#3$}}\kern-0.5\wd0}
\def\bint{\Xint-}
\def\dashint{\Xint{\raise4pt\hbox to7pt{\hrulefill}}}
\def\dashiint{\bint\kern-0.15cm\bint}

 \def\Xiint#1{\mathchoice
     {\XXiint\displaystyle\textstyle{#1}}%
     {\XXiint\textstyle\scriptstyle{#1}}%
     {\XXiint\scriptstyle\scriptscriptstyle{#1}}%
     {\XXiint\scriptscriptstyle\scriptscriptstyle{#1}}%
     \!\iint}
 \def\XXiint#1#2#3{\setbox0=\hbox{$#1{#2#3}{\iint}$}
     \vcenter{\hbox{$#2#3$}}\kern-0.5\wd0}
 \def\biint{\Xiint{-\!-}}

\def\XXiiint#1#2#3{\setbox0=\hbox{$#1{#2#3}{\iint}$}
    \vcenter{\hbox{$#2#3$}}\kern-0.5\wd0}

\newtheorem{proposition}{Proposition}[section]
\newtheorem{theorem}{Theorem}[section]

\newtheorem{lemma}{Lemma}[section]
\newtheorem{corollary}{Corollary}[section]
\newtheorem{remark}{Remark}[section]
\newtheorem{definition}{Definition}[section]

\numberwithin{equation}{section}
\numberwithin{theorem}{section}
\numberwithin{proposition}{section}
\numberwithin{lemma}{section}
\numberwithin{remark}{section}
\title[Nonlocal intrinsic Harnack estimates]{Nonlocal intrinsic Harnack estimates}

\author[N. Liao]{Naian Liao}
\address{Fachbereich Mathematik, Universit\"at Salzburg,
Hellbrunner Str. 34, 5020 Salzburg, Austria}{}
\email{naian.liao@plus.ac.at}

\begin{document}



\begin{abstract}
 We extend DiBenedetto's intrinsic parabolic Harnack estimates to a class of integro-differential equations with measurable kernels  modeled on the fractional $p$-Laplacian. More remarkably, we establish intrinsic, representation type nonlocal estimates and reveal that nonnegative local solutions abide by a strong maximum principle even for $p>2$, a property unique to nonlocal problems. Analogous results hold true for nonlocal equations of porous medium type.
\end{abstract}  

\date{\today}

\maketitle

%%%%%%%%%%%%%%
\section{Introduction}%\label{S:intro}
In 1988 DiBenedetto \cite{DB-88} established an intrinsic Harnack type inequality for the evolution of degenerate $p$-Laplacian:
\begin{equation}\label{Eq:p-Laplace}
  \pl_t u -\dvg(|\nabla u|^{p-2}\nabla u)=0,\quad p>2 . 
\end{equation}
Precisely, if $u$ is a nonnegative, continuous solution satisfying $u(x_o,t_o)>0$ and if $B_{4\rho}(x_o)\times (t_o- [\sig u(x_o,t_o)]^{2-p}(4\rho)^p, t_o+ [\sig u(x_o,t_o)]^{2-p}(4\rho)^p]$ is included in the domain under consideration, then 
\[
u(x_o,t_o)\leqslant \boldsymbol{\gm} \inf_{B_\rho(x_o)} u(\cdot, t_o+[\sig u(x_o,t_o)]^{2-p}\rho^p)
\]
holds true, where $\boldsymbol{\gm}$ and $\sig$ depend only on $p$ and the space dimension.
The thrust is that equation~\eqref{Eq:p-Laplace} is inhomogeneous in $u$, and the {\it intrinsic time scaling} restores the homogeneity. In other words, nonnegative solutions to \eqref{Eq:p-Laplace} behave like caloric functions in their own intrinsic geometry.
His original approach was based on the maximum principle and comparison functions constructed as variants of the Barenblatt self-similar solution. In a sense, such an approach could be regarded as paralleling that of Hadamard and Pini for caloric functions.
However, a fundamental gap remained between the elliptic theory and the parabolic one. Indeed, elliptic Harnack estimates were known to be  purely structural facts, independent of any comparison argument. This gap was eventually filled by DiBenedetto, Gianazza and Vespri \cite{DBGV-acta} twenty years later. See  the monographs~\cite{DB,DBGV-mono, Urbano-08} for a whole account.

\subsection{The evolution of $(s,p)$-Laplacian}
Over the past few years, we have witnessed a growing interest in the fractional $p$-Laplacian, sometimes also called $(s,p)$-Laplacian. The study is partly motivated by a variational consideration, like the one for the $p$-Laplacian.  Markedly, a basic DeGiorgi-Nash-Moser theory has been established for it, and an evolutionary nonlocal theory is emerging; see~\cite{Adi, Byun-ampa, Caff-Vass-11, CCI-26, ChenJY, Cozzi, DKP-1, DKP-2, Kass-13, Kass-09, Kass-20, Kass-23, Liao-cvpd-24, Liao-mod-24,Liao-DD-24} for a fraction of the vast literature. Dealing with a measurable integral kernel is the common feature of these papers.

In this work, we continue the study set out in \cite{Liao-cvpd-24} regarding
a class of evolution equations that feature a nonlocal operator of $(s,p)$-Laplacian type:
\begin{equation}\label{Eq:1:1}
\pl_t u + \mathscr{L} u=0\quad\text{weakly in}\>\> E_T:=E\times(0,T],
\end{equation}
%where $E_T=E\times(0,T]$ 
for some open set $E\subset\rn$ and some $T>0$. 
The nonlocal operator $\mathscr{L}$ is defined by
\begin{equation}\label{Eq:1:2}
\mathscr{L}u(x,t)={\rm p.v.}\int_{\rn} \mathsf{K}(x,y,t) |u(x,t) - u(y,t) |^{p-2}  (u(x,t) - u(y,t) )\,\dy,
\end{equation}
for some $p>2$,
whereas %${\rm P.V.}$ denotes the principle value of the integral, whereas 
 the kernel $\mathsf{K}:\rn\times\rn\times(0,T]\to [0,\infty)$ is measurable and satisfies the following condition uniformly in $t$:
\begin{equation}\label{Eq:K}
\frac{C_o}{|x-y|^{N+sp}}\leqslant \mathsf{K}(x,y,t)\equiv \mathsf{K}(y,x,t)\leqslant \frac{C_1}{|x-y|^{N+sp}}\quad\text{a.e.}\>\> x,\,y\in\rn,
\end{equation}
for some positive $C_o$, $C_1$ and $s\in(0,1)$. 
In \cite{Liao-mod-24} we have proved that local weak solutions are locally continuous (not H\"older continuous in general). Based on that, we open up the nonlocal intrinsic Harnack type estimates.
\begin{theorem}[Harnack type estimates]\label{Thm:1}
   Let  $u$ be a continuous, local weak solution to the $(s,p)$-parabolic type equations~\eqref{Eq:1:1}--\eqref{Eq:K} with $p>2$, in the sense of Definition~\ref{Def:sol}.
Assume that $u\geqslant0$ in $E_T$ and  $u_o:=u(x_o,t_o)>0$ for some $(x_o,t_o)\in E_T$. There exist constants $\boldsymbol{\gm}_{\mathsf{H}}>1$ and $\kappa,\,\sig\in(0,1)$ depending on $s$, $p$, $N$, $C_o$, and $C_1$, such that if  
\begin{equation}\label{Thm:condition:1}
    B_{100\rho}(x_o)\times (t_o-2\theta_{\sig}(8\rho)^{sp}, t_o+ 2\theta_{\sig} (8\rho)^{sp} ]\subset E_T,
\end{equation} 
 where $\theta_{\sig}:=(\sig u_o)^{2-p}$ and if
\begin{equation}\label{Thm:condition:2}
    \boldsymbol{\gm}_{\mathsf{H}}\int_{t_o-2\theta_{\sig}(8\rho)^{sp}}^{t_o+2\theta_{\sig}(8\rho)^{sp}}\int_{\rn\setminus B_{\rho}(x_o)}\frac{u_-^{p-1}(x,t)}{|x-x_o|^{N+sp}}\,\dx\dt\leqslant\sig u_o,
\end{equation}
then we have the following two conclusions.

\begin{itemize}
    \item[(i)] The local, intrinsic Harnack estimate holds true:
    \begin{equation}\label{Thm:conclusion:1}
        \sig\sup_{B_\rho(x_o)}u (\cdot, t_o-\kappa\theta_{1}\rho^{sp} ) \leqslant u_o \leqslant \sig^{-1}\inf_{B_\rho(x_o)}u (\cdot, t_o+\kappa\theta_{\sig}\rho^{sp} ).
    \end{equation}
    \item[(ii)] The global,  intrinsic representation type estimate holds true:
    \begin{align}\label{Thm:conclusion:2}\nonumber
    \frac{\kappa\sig}{\boldsymbol{\gm}_{\mathsf{H}}}\,\int^{t_o-\kappa\theta_1\rho^{sp}}_{t_o-\kappa\theta_1(2\rho)^{sp}}\int_{\rn}&\frac{u_+^{p-1}(x,t)}{ (\rho+|x-x_o|)^{N+sp}}\,\dx\dt  \\
    &\leqslant u_o \leqslant  \boldsymbol{\gm}_{\mathsf{H}}\,\int^{t_o}_{t_o-2^{p-1}\theta_1\rho^{sp}}\int_{\rn}\frac{u_+^{p-1}(x,t)}{ (\rho+|x-x_o|)^{N+sp}}\,\dx\dt.
\end{align}
\end{itemize}
\end{theorem}

\begin{remark}\upshape
    The parameter $\sig$ once determined by $s$, $p$, $N$, $C_o$, and $C_1$, can be made smaller in these estimates if needed.
\end{remark}

\begin{remark}\upshape
No initial or boundary data will be  prescribed.
    The Harnack type estimates are local and structural facts that do not rely on solving any Cauchy or Dirichlet problem, nor on any comparison argument, nor on time monotonicity estimates. 
    Equations \eqref{Eq:1:1}--\eqref{Eq:K} are used only once to derive  {\it local energy estimates}. Then, all results are consequences solely of the energy estimates.
    On the other hand, Theorem~\ref{Thm:1} could be used to study the local behavior of solutions near the boundary.
\end{remark}

\begin{remark}\label{Rmk:continuity}\upshape
    Continuity of solutions proven in \cite{Liao-mod-24} gives pointwise meaning to $u(x_o,t_o)$ that makes our formulation of Harnack type estimates legitimate. It also provides crucial topological information to carry out the proof. However, we do not need a quantitative modulus of continuity here. In turn,  the local, intrinsic Harnack estimate~\eqref{Thm:conclusion:1} can be used to simplify the argument in \cite[\S~5]{Liao-mod-24}; the gist is that qualitative continuity of solutions plus \eqref{Thm:conclusion:1} yields a quantitative modulus of continuity.
\end{remark}

\begin{remark}\upshape
Condition~\eqref{Thm:condition:1} defines the nature of the results: they are local estimates inside an intrinsic cylinder, whereas condition~\eqref{Thm:condition:2} quantifies the influence of $u_-$ outside the domain. This condition becomes vacuous if $u$ is globally nonnegative but in general unavoidable as already observed in the simplest, elliptic setting by Kassmann \cite{Kass-11}.
Moreover, in the linear parabolic case $p=2$, estimate~\eqref{Thm:conclusion:1} parallels  \cite[Eq.~(1.1)]{Kass-23}, whereas the left/right-hand side of  \eqref{Thm:conclusion:2} formally recovers Eqs.~(1.2)/(1.10) of that paper. The main difference lies in the pointwise formulation we adopt, in order to accommodate the intrinsic geometry inherent in the nonlinear structure.
    However,  we will not study stability as $p\to2$ or as $s\to1$; rather, we emphasize the peculiarity that arises in the nonlinear, nonlocal setting.
\end{remark}

The local, intrinsic Harnack estimate \eqref{Thm:conclusion:1} can be regarded as a nonlocal analog of  Theorem~1.1 in \cite[Chapter~5]{DBGV-mono}, whereas the representation type estimate \eqref{Thm:conclusion:2} is a purely nonlocal fact. Both of them shed light on how nonnegative solutions behave in the vicinity of their vanishing set. As a matter of fact, estimate \eqref{Thm:conclusion:1} yields a {\it sub-potential estimate} of nonnegative global solutions via a chain argument; see \cite{DGV-07} and Proposition~6.2 of \cite[Chapter~5]{DBGV-mono}. Essentially, such an estimate shows that any nonnegative solution is bounded from below by a variant of the Barenblatt solution that features a moving boundary--a characteristic phenomenon of slow/degenerate diffusion equations like \eqref{Eq:p-Laplace}. Given that one might expect a similar profile for the evolution of $(s,p)$-Laplacian as well.

Strikingly, we will demonstrate through  estimate \eqref{Thm:conclusion:2} that nonnegative local solutions propagate positivity instantly like caloric functions, and hence, a moving boundary does not exist in general. In this sense, estimate \eqref{Thm:conclusion:2} can also be deemed as a nonlocal mean value estimate.
\begin{theorem}[Strong maximum principle]\label{Thm:2}
       Let  $u$ be a continuous, local weak solution to the $(s,p)$-parabolic type equations~\eqref{Eq:1:1}--\eqref{Eq:K} with $p>2$ in the sense of Definition~\ref{Def:sol}.
Assume that $u\geqslant0$ a.e. in $\rn\times(0,T)$ and $(x_o,t_o)\in E\times(0, T)$. If $u(x_o,t_o)>0$, then $u>0$ in $E\times[t_o,T)$. Alternatively, if $u(x_o,t_o)=0$, then $u=0$ a.e. in $\rn\times(0,t_o]$. 
\end{theorem}

Let us stress two points regarding Theorem~\ref{Thm:2}. First, this is a local property, independent of any initial or boundary data. The only global information we impose is the non-negativity of solutions, which cannot be removed in general, even in the simplest, elliptic setting.
Second, this kind of property is false for the degenerate $p$-parabolic equation \eqref{Eq:p-Laplace}; in addition, even if it resembles that of caloric functions, however, it disregards the connectedness of $E$. As already indicated, such property cannot be seen from the local estimate~\eqref{Thm:conclusion:1} alone; it is estimate~\eqref{Thm:conclusion:2} that captures this essentially nonlocal character. Finally, let us point out that our approach to Theorem~\ref{Thm:2}, can also be extended to general space-time domains.

\subsection{Method of proof}\label{S:method}
A basic regularity theory of DeGiorgi-Nash-Moser kind typically consists of establishing weak Harnack estimates of nonnegative super-solutions, H\"older estimates of solutions, and Harnack estimates of nonnegative solutions. In the context of {\it linear} parabolic equations, these issues are of comparable difficulty; establishing either of them requires independent measure theoretical arguments. However, when it comes to {\it nonlinear} parabolic equations, more refined perspectives are required for each of them, and the technical demands are on another level.

We provide a new perspective to the proof of nonlocal Harnack's inequalities in general, whether in the elliptic or parabolic setting. Namely, we pursue a direct proof of Harnack's inequality for solutions without first establishing a weak Harnack estimate for super-solutions, the latter being left as an interesting open problem. The scheme evolves out of the one laid out in our previous papers \cite{Liao-cvpd-24, Liao-mod-24, Liao-DD-24}.
Here, we single out two crucial, novel components.

$\bullet$~{\it Expansion of positivity}. Roughly speaking, it translates measure theoretical positivity of nonnegative super-solutions to a pointwise one. This property lies at the heart of any form of Harnack estimate, no matter the setting is elliptic or parabolic, divergence form or non-divergence form, local or nonlocal. 

However, turning this principle into an effective estimate is far from straightforward.
In order to establish an intrinsic version of expansion of positivity for the $p$-parabolic type equations, an exponential-time-shift technique is ingeniously introduced in \cite{DBGV-acta}. Loosely speaking, the technique stretches time at an exponential scale and compensates loss of information in the time direction due to inhomogeneity of the $p$-parabolic type equations. In view of a similar inhomogeneity of the $(s,p)$-parabolic type equations, it is conceivable that this technique could be applied here.
Indeed,  this has been done in \cite{Adi, Misawa-25}. Nonetheless, we think their results cannot be applied to prove the Harnack estimates, as they impose a stronger tail condition that will inevitably break the delicate balance between the local and nonlocal contributions; see further discussion on {\it Tail control} below. Here, we will prove a new version that conforms to the natural tail condition \eqref{Eq:global-int}; see~Proposition~\ref{Prop:expansion}. 

More importantly, and somehow unexpectedly, we are able to dispense with the exponential-time-shift technique and demonstrate a new perspective: the fractional Laplacian owns a special mechanism of propagating positivity that compensates the loss of information in the time direction due to the ``degenerate scaling'' inherent in the $(s,p)$-parabolic equation. The strong maximum principle is a manifestation of this mechanism. Such a point has been overlooked in the literature; we believe this idea could be generalized to attack more complicated nonlinear objects. 

Moreover, expansion of positivity has further important applications. Let us mention two instances here. First, it can be used to streamline the proof of the modulus of continuity given in \cite[\S~5]{Liao-mod-24} that instead exploited the complicated {\it sliding-cylinder method} of DiBenedetto. Needless to say, the same idea could also be used to fruitfully attack the continuity issue of other nonlinear, nonlocal equations; see~\S~\ref{S:appl} for discussions. Second, it lays the foundation of proving a weak Harnack type inequality for nonnegative super-solutions, which has been bypassed in the current work.

$\bullet$~{\it Tail control}.
Expansion of positivity and weak Harnack estimates  essentially depict local behavior of nonnegative super-solutions. To obtain Harnack estimates, they are insufficient; we need to acquire fine control of solutions' long-range behavior.

To illustrate the  point, let us recall the approach to the elliptic nonlocal Harnack estimates that appears in \cite{DKP-2}. There are three key elements.
\begin{itemize}
    \item[(i)] An interpolated $L^\infty$-estimate for sub-solutions: 
    \begin{align}\label{Eq:elliptic:1}
        \sup_{B_1}u \leqslant \dl\, \mathrm{Tail}[u; B_1] + \boldsymbol{\gm}(\dl)\|u\|_{L^p(B_2)},\quad\forall\, \dl\in(0,1)
    \end{align}
    \item[(ii)] A weak Harnack estimate for super-solutions: 
    \begin{align}\label{Eq:elliptic:2}
        \boldsymbol{\gm}\inf_{B_1}u \geqslant \|u\|_{L^p(B_2)} 
    \end{align}
    \item[(iii)] A ``$\mathrm{Tail}\leqslant \sup$'' type estimate for solutions: 
    \begin{align}\label{Eq:elliptic:3}
        \mathrm{Tail}[u; B_1]\leqslant \boldsymbol{\gm} \sup_{B_1}u 
    \end{align}
\end{itemize}
As said earlier, the weak Harnack estimate \eqref{Eq:elliptic:2} is a local property since it does not involve the tail, whereas a tail term, differently from local equations, appears in the $L^{\infty}$-estimate \eqref{Eq:elliptic:1}. The approach adopted in \cite{DKP-2} exploits estimate~\eqref{Eq:elliptic:3} and applies an interpolation argument to \eqref{Eq:elliptic:1} thanks to the parameter $\dl$, which eventually leads to 
\[
\sup_{B_1}u \leqslant   \boldsymbol{\gm} \|u\|_{L^p(B_2)}
\]
and completes the proof by joining it with \eqref{Eq:elliptic:2}.

This scheme breaks down in the parabolic setting, since an $L^\infty$-estimate that interpolates the natural parabolic tail and the local norm of $u$ generally fails to hold; see~\cite[\S~5.2]{Liao-DD-24} for an example and also the discussion following \cite[Theorem~1.9]{Kass-23}.
Instead, we have proposed a different point of view in \cite{Liao-DD-24} to read the local positivity of super-solutions out of their tail. More explicitly, we have established a measure theoretical estimate
\begin{equation}\label{Eq:elliptic:3.5}
 \frac{|\{u \leqslant  k \}\cap B_1|}{|B_1|}\leqslant  \frac{ \boldsymbol\gm k^{p-1}}{[{\rm Tail}[u; B_1]]^{p-1}} 
\end{equation}
for super-solutions; see~\cite[Lemma~7.3]{Liao-DD-24}.
Such an estimate replaces \eqref{Eq:elliptic:3}, bridges relevant local and nonlocal quantities with a finer, measure theoretical perspective, and directly yields that
\begin{align} \label{Eq:elliptic:4}
        \boldsymbol{\gm}\inf_{B_1}u \geqslant \mathrm{Tail}[u; B_1] 
    \end{align}
for super-solutions, which was previously unknown. By plugging \eqref{Eq:elliptic:2} and \eqref{Eq:elliptic:4} back in \eqref{Eq:elliptic:1}, we conclude the desired Harnack estimate without an explicit interpolation argument.

$\bullet$~{\it The road map}.
We will adapt these ideas in this work. In addition to  a proper version of \eqref{Eq:elliptic:3.5}, to confront the $(s,p)$-parabolic type equations, we need to develop measure theoretical counterparts of estimates~\eqref{Eq:elliptic:1} and \eqref{Eq:elliptic:2} as well. In particular, the expansion of positivity that we discussed previously will play the role of \eqref{Eq:elliptic:2}, whereas the role of \eqref{Eq:elliptic:1} is taken by the so-called {\it critical density lemma}, or equivalently by the right-hand estimate in \eqref{Thm:conclusion:2}. Once we have all the measure theoretical modules at our fingertips, we prove the Harnack estimates by ``juggling'' these modules as follows. 

The first, crucial step lies in establishing the forward estimate (right-hand side) of \eqref{Thm:conclusion:1}. Roughly speaking, if we know $u(0,0)=1$, then we can show there must be a tiny space ball below $t=0$, where $u$ is large in the measure theoretical sense by exploiting the critical density lemma together the analog of \eqref{Eq:elliptic:3.5}. Subsequently, this measure theoretical positivity will be translated into pointwise estimates of later times in a spatially expanded fashion, via the expansion of positivity, yielding the forward estimate of \eqref{Thm:conclusion:1}.

In the second step, we show that the backward estimate (left-hand side) of \eqref{Thm:conclusion:1} is in fact encoded in the forward estimate established in the first step. To wit, this step consists of a purely local argument.

The left-hand estimate in \eqref{Thm:conclusion:2} relies on what has been shown in the second step; it can be viewed as sort of ``weak Harnack estimate for solutions.''

Lastly, we show the right-hand estimate in \eqref{Thm:conclusion:2}, which is essentially an $L^\infty$-estimate. This step is independent of the previous three steps. In fact, this estimate can  replace the role of the critical density lemma in the first step. Therefore, it could also be labeled as the zeroth step.
%%%%%%
\subsection{Definitions and notation}%\label{S:notion}

Our notion of solution is a local concept that does not involve any initial or boundary data.

\begin{definition}[Local solution]\label{Def:sol}
A measurable function $u:\,\rn\times(0,T]\to\rr$ satisfying
\begin{equation}  \label{Eq:1:3p}
	u\in C_{\loc} (0,T;L^2_{\loc}(E) )\cap L^p_{\loc} (0,T; W^{s,p}_{\loc}(E) )
\end{equation}
is a local, weak sub(super)-solution to the $(s,p)$-parabolic type equations \eqref{Eq:1:1}--\eqref{Eq:K}, if for   every 
$[t_1,t_2]\subset (0,T]$, we have
\begin{equation}\label{Eq:global-int}
\int_{t_1}^{t_2}\int_{\rn}\frac{|u(x,t)|^{p-1}}{1+|x|^{N+sp}}\,\dx\dt<\infty
\end{equation}
and
\begin{equation}\label{Eq:1:4p}
\begin{aligned}
	\int_{E} \vp\, u\,\dx\bigg|_{t_1}^{t_2}
	-\int_{t_1}^{t_2}\int_{E}  \pl_t\vp\, u\, \dx\dt
	+\int_{t_1}^{t_2} \mathscr{E} (u(\cdot, t), \vp(\cdot, t) )\,\dt
	\leqslant(\geqslant)0
\end{aligned}
\end{equation}
where  
\begin{align*}
    \mathscr{E}(t):=\int_{\rn}\int_{\rn}\mathsf{K}(x,y,t) \Phi_p(u(x,t) - u(y,t) ) (\vp(x,t) - \vp(y,t) )\,\dy\dx
\end{align*}
and $\rr\ni\xi\mapsto\Phi_p(\xi):=|\xi|^{p-2}\xi$,
for all nonnegative testing functions
\begin{equation} \label{Eq:test-function}
\vp \in W^{1,2}_{\loc} (0,T;L^2(E) )\cap L^p_{\loc} (0,T;W_o^{s,p}(E)).
\end{equation}

A function $u$ that is both a local weak sub-solution and a local weak super-solution
to \eqref{Eq:1:1}--\eqref{Eq:K} is a local weak solution.
%When $T=\infty$, we define $u$ to be a local weak solution for every finite $T$.
\end{definition}
 
 The condition \eqref{Eq:global-int} gives rise to  the notation
\begin{equation}\label{Eq:tail}
{\rm Tail}[u; Q(R,S)]:= \int_{t_o-S}^{t_o} \int_{\rn\setminus B_R(x_o)}\frac{|u(x,t)|^{p-1}}{|x-x_o|^{N+sp}}\,\dx \dt.
\end{equation}
Here, and in what follows, 
we will use  the symbols 
\begin{equation*}
\left\{
\begin{aligned}
&(x_o,t_o)+Q(R,S):=B_R(x_o)\times (t_o-S,t_o]\\[5pt]
&(x_o,t_o)+Q_\rho(\theta):=B_{\rho}(x_o)\times(t_o-\theta\rho^{sp},t_o]\\[5pt]
&(x_o,t_o)+\mathbb{Q}(R,S):=B_R(x_o)\times B_R(x_o)\times (t_o-S,t_o]\\[5pt]
&(x_o,t_o)+\mathbb{Q}_\rho(\theta):=B_{\rho}(x_o)\times B_{\rho}(x_o)\times(t_o-\theta\rho^{sp},t_o]
\end{aligned}\right.
\end{equation*} 
to denote (backward) cylinders, where $B_\rho(x_o)$ denotes the ball of radius $\rho$ and center $x_o$ in $\rn$.
The vertex $(x_o,t_o)$ has been omitted from the cylinder in \eqref{Eq:tail} for simplicity. 
If $\theta=1$, it will also be omitted. When the context is unambiguous, we will apply these conventions.

Throughout this note,  we let 
$$\mathfrak{data}:=\{s, p, N, C_o, C_1\},$$ and we use $\boldsymbol\gm=\boldsymbol{\gm}(\mathfrak{data})$ as a generic positive constant that can be determined by the $\mathfrak{data}$ only and that can vary from one estimate to another.
%%%%%%%%%%%%%%%%%%%%%%%%%%%%%%
\subsection{Organization}
The article is structured as follows. In \S~\ref{S:energy}, we present the energy estimates.
Then, we collect some useful preparatory modules in \S~\ref{S:prelim} for the proof of the Harnack estimates, whereas expansion of positivity is singled out in \S~\ref{S:exp-pos}. Then, the proof of Theorem~\ref{Thm:1} and Theorem~\ref{Thm:2} is given in \S~\ref{S:Thm:1-proof} and in \S~\ref{S:Thm:2-proof}, respectively. Finally, we discuss applications to other nonlocal equations in \S~\ref{S:appl}.

%\

%%%%%%%%%%%%%%%%%%%%%%%%%%%%%

%%

%%%%%%%%%%%%%%%%%%%%%%%%%%%%%%
\section{Local energy estimates}\label{S:energy}

The energy estimates are taken from \cite[Proposition~2.1]{Liao-mod-24}.   The idea of using a time-dependent truncation level $k(t)$ is taken from \cite{Kass-23}.

\begin{proposition}\label{Prop:2:1}
	Let $u$ be a  local weak sub(super)-solution to \eqref{Eq:1:1}--\eqref{Eq:K} in $E_T$, and
	let $k(\cdot)$ be  absolutely continuous  in $(0,T)$.
	There exists a constant $\boldsymbol \gm$ depending on $C_o$, $C_1$, and $p$, such that
 	for all cylinders $Q(R,S) \subset E_T$,
 	and  any nonnegative, Lipschitz function $\zeta(x,t)$ in $Q_{R,S}$ such that $\mathrm{spt}(\z(\cdot, t))\subset B_R$ for any $t \in (t_o - S, t_o)$,  we have
\begin{align*}
	&\iiint_{\mathbb{Q}(R,S)} \min\{\z^p(x,t),\,\z^p(y,t)\} \frac{|w_\pm(x,t) - w_\pm(y,t)|^p}{|x-y|^{N+sp}}\,\dx\dy\dt\\
	&\qquad+\iint_{Q(R,S)} \z^p w_{\pm}(x,t)\,\dx\dt \bigg[\int_{ \rn}  \frac{w^{p-1}_\mp(y,t)}{|x-y|^{N+sp}}\,\dy\bigg]\\
    &\qquad+\int_{B_R}\z^p w^2_{\pm}(x,t)\,\dx\bigg|_{t_o-S}^{t_o}\\
	&\quad\leqslant
	\boldsymbol\gm\iiint_{\mathbb{Q}(R,S)}\max\{w^p_{\pm}(x,t),\, w^p_{\pm}(y,t)\} \frac{|\z(x,t) - \z(y,t)|^p}{|x-y|^{N+sp}}\,\dx\dy\dt\\
	&\qquad+\boldsymbol\gm\iint_{Q(R,S)}\int_{\rn\setminus B_R} \z^p w_{\pm}(x,t)\frac{ w_{\pm}^{p-1}(y,t)}{|x-y|^{N+sp}}\,\dy\dx\dt %\bigg(\essup_{\substack{x\in\supp\z(\cdot, t)\\t\in(t_o-S,t_o)}} \int_{\rn\setminus B_R}\frac{ w_{\pm}^{p-1}(y,t)}{|x-y|^{N+sp}}\,\dy\bigg)
	\\
	&\qquad \mp 2\iint_{Q(R,S)}  k'(t) \z^pw_{\pm}(x,t)\,\dx\dt + \iint_{Q(R,S)} |\pl_t\z^p|w_{\pm}^2(x,t)\,\dx\dt.  %\\
%	&\phantom{\leqslant\,}
%	+\int_{B_R(x_o)\times \{t_o-S\}} \z^p \mathfrak g_\pm (u,k)\,\dx.
\end{align*}
Here, we have denoted $w(x,t)=u(x,t)-k(t)$ for simplicity.
\end{proposition}
In the above statement, we took the upper sign of $\pm/\mp$ for sub-solutions and the lower sign for super-solutions. This convention will be applied to other statements in the article.

As already indicated in the introduction, the Harnack estimates are consequences solely of the energy estimates. Before proving them, we first collect some direct consequences of the energy estimates.

%%%%%%%%%%%%%%%%%%%%%%%%%%%%%%%%%%%

\section{Preparatory tools} \label{S:prelim}

Our general program has not changed.
Indeed, the first four results of this section  collect those from \cite{Liao-cvpd-24, Liao-mod-24}, whereas the last three results generalize those from \cite{ Liao-DD-24}. Moreover, all results in this section hold for any $p>1$. %with one exception--Lemma~\ref{Lm:DG:2}, which will improve the first lemma.

To state the first result of this section, let us introduce the reference cylinder
 $\mathcal{Q}:=B_R(x_o)\times(T_1,T_2]\subset E_T$.
Suppose the quantities $\boldsymbol\mu^{\pm}$ and $\boldsymbol\om$ satisfy
\begin{equation*}
	\boldsymbol\mu^+\geqslant\essup_{\mathcal{Q}}u,
	\quad 
	\boldsymbol\mu^-\leqslant\essinf_{\mathcal{Q}} u,
	\quad
	\boldsymbol\om\geqslant\boldsymbol\mu^+-\boldsymbol\mu^-.
\end{equation*}
With the above quantities, we state the first result that is actually \cite[Lemma 3.1]{Liao-mod-24}; the dependence of $\nu$ can be traced in the same proof. 
\begin{lemma}[Critical density lemma]\label{Lm:DG:1}
 Let $u$ be a locally bounded, local weak sub(super)-solution to \eqref{Eq:1:1}--\eqref{Eq:K} in $E_T$.
 For some parameters $a,\, \dl,\,\xi\in(0,1)$ and $\rho\in(0,\frac12R]$, set $\theta=\dl(\xi\boldsymbol\om)^{2-p}$, and assume $ (x_o,t_o)+Q_\varrho(\theta) \subset \mathcal{Q}$.
There exist $\widetilde{\boldsymbol\gm}>1$ depending only on 
 the  $\mathfrak{data}$ and   $\nu\in(0,1)$ depending on the $\mathfrak{data}$, $a$, and $\dl$, such that if
\begin{equation*}
	|\{
	\pm(\boldsymbol \mu^{\pm}-u)\leqslant  \xi\boldsymbol\om\}\cap  (x_o,t_o)+Q_{\varrho}(\theta)|
	\leqslant
	\nu|Q_{\varrho}(\theta)|,
\end{equation*}
then either
\begin{equation*}%\label{Eq:DG:alt}
\widetilde{\boldsymbol\gm} 
{\rm Tail}[(u - \boldsymbol \mu^{\pm})_\pm; \mathcal{Q}]
%\int^{T_2}_{T_1}\int_{\rn\setminus B_R(x_o)}\frac{ \big(u - \boldsymbol \mu^{\pm}\big)_{\pm}^{p-1}}{|x|^{N+sp}}\,\dx\dt
>\tfrac12 a\xi\boldsymbol\om
\end{equation*}
or
\begin{equation*}
	\pm(\boldsymbol\mu^{\pm}-u)\geqslant \tfrac12 a\xi\boldsymbol\om
	\quad
	\mbox{a.e.~in $(x_o,t_o)+ Q_{\frac{1}2\varrho}(\theta)$.}
\end{equation*}
Moreover, we have  
\begin{equation}\label{Eq:nu-dep}
    \nu=\frac{(1-a)^{q_1}\dl^{q_2}}{\widetilde{\boldsymbol\gm}}
\end{equation} 
for some $q_1,\,q_2>0$ depending on $p$, $N$ and $s$.
\end{lemma}

The second result comes from \cite[Lemma 3.2]{Liao-mod-24} by taking $\mu^-=0$; here and in what follows, we only need the case of nonnegative super-solutions.
\begin{lemma}\label{Lm:DG:initial:1}
Let $u$ be a local weak super-solution to \eqref{Eq:1:1}--\eqref{Eq:K} that is nonnegative in $\mathcal{Q}$, and let $k>0$ be a parameter. 
There exist $\dl_o\in(0,1)$ and $\widetilde{\boldsymbol\gm}>1$ depending only on the $\mathfrak{data}$, such that if
\[
u(\cdot, t_o) \geqslant k \quad\text{ a.e. in } B_{\rho}(x_o),
\]
then either
\begin{equation*}
\widetilde{\boldsymbol\gm}
{\rm Tail}[u_-; \mathcal{Q}]
%\int^{T_2}_{T_1}\int_{\rn\setminus B_R(x_o)}\frac{ \big(u - \boldsymbol \mu^{\pm}\big)_{\pm}^{p-1}}{|x|^{N+sp}}\,\dx\dt
>k
\end{equation*}
or
\[
u\geqslant \tfrac14k\quad\text{ a.e. in }B_{\frac12\rho}(x_o)\times(t_o, t_o+\dl_o k^{2-p}\rho^{sp}],
\]
provided $B_{2\rho}(x_o) \times(t_o,t_o+\dl_o k^{2-p}\rho^{sp}]\subset\mathcal{Q}$.
\end{lemma}

 The third result is comes from \cite[Lemma 3.3]{Liao-mod-24}
%The next result propagates measure forward in time.
\begin{lemma}\label{Lm:meas-propogation}
 Let $u$ be a local weak super-solution to \eqref{Eq:1:1}--\eqref{Eq:K} that is nonnegative in $\mathcal{Q}$.
Introduce parameters $k>0$ and $\al\in(0,1)$. There exists  $\dl \in(0,1)$ depending only on the $\mathfrak{data}$ and $\al$, such that if
	\begin{equation*}
	|\{
		u(\cdot, t_o)\geqslant k
		\}\cap B_{\varrho}(x_o)|
		\geqslant\al |B_{\varrho}|,
	\end{equation*}
	then %for $\widetilde{\boldsymbol\gm}=1/\dl$, 
	either 
$$ 
\frac1\dl 
{\rm Tail}[u_-; \mathcal{Q}]
%\int^{T_2}_{T_1}\int_{\rn\setminus B_R}\frac{ \big(u - \boldsymbol \mu^{\pm}\big)_{\pm}^{p-1}}{|x|^{N+sp}}\,\dx\dt
>k
$$ 
	or
	\begin{equation*}%\label{Eq:3:1}
	|\{
	u(\cdot, t)\geqslant \tfrac14\al k\} \cap B_{\varrho}(x_o)|
	\geqslant\tfrac{1}2\al |B_\varrho|
	\quad\mbox{ for all $t\in(t_o,t_o+\dl k^{2-p}\varrho^{sp}]$,}
\end{equation*}
provided $B_{2\rho}(x_o) \times(t_o,t_o+\dl k^{2-p}\rho^{sp}]\subset\mathcal{Q}$. Moreover, we have %$\varep=\frac14\al$ and 
$\dl=\al^{p+N+1}/\boldsymbol{\gm}$.
\end{lemma}

The last result we quote from \cite[Lemma 3.4]{  Liao-mod-24} is a measure shrinking lemma.
%%%%%%%%
\begin{lemma}\label{Lm:meas-shrink:1}
 Let $u$ be a  local weak super-solution to \eqref{Eq:1:1}--\eqref{Eq:K} that is nonnegative in $\mathcal{Q}$.
 For  some parameters $k>0$, $\al\in (0,1)$ and $\dl,\, \sig\in(0,\tfrac12)$, let $\theta=\dl k^{2-p}$. Suppose that
	\begin{equation*}
	|\{
		u(\cdot, t)\geqslant k
		\}\cap B_{\varrho}(x_o)|
		\geqslant\al |B_{\varrho}|\quad\mbox{ for all $t\in(t_o-\theta\varrho^{sp}, t_o]$.}
	\end{equation*}
There exists
 $\boldsymbol \gm>1$ depending only on the $\mathfrak{data}$ and independent of $\{\al, \dl, \sig\}$,  such that 
 %for $\widetilde{\boldsymbol\gm}=1/\dl$,
  either 
  $$ 
\frac1\dl  
{\rm Tail}[u_-; \mathcal{Q}]
%\int^{T_2}_{T_1}\int_{\rn\setminus B_R}\frac{ \big(u - \boldsymbol \mu^{\pm}\big)_{\pm}^{p-1}}{|x|^{N+sp}}\,\dx\dt
>k
  $$ 
  or
\begin{equation*}
	|\{
	 u \leqslant \sig k \}\cap (x_o,t_o)+Q_{\rho}(\theta)|
	\leqslant \boldsymbol\gm \frac{\sig^{p-1}}{\dl\al} |Q_{\rho}(\theta)|,
\end{equation*}
provided $(x_o,t_o)+Q_{2\rho}(\theta)\subset\mathcal{Q}$.
\end{lemma}

The next three results generalize some linear ones from \cite{Liao-DD-24}.
The first generalization asserts that the smallness in measure at an instant yields pointwise estimate of later times. As such it can be viewed as a refinement of Lemma~\ref{Lm:DG:1} of the current work. For $p=2$, this kind of result has appeared in  \cite[Lemma~3.2]{Liao-DD-24}.
%%%%%%%%
\begin{lemma}\label{Lm:DG:2}
    Let $u$ be a local weak super-solution to \eqref{Eq:1:1}--\eqref{Eq:K} that is nonnegative in $\mathcal{Q}$.
 For some $k>0$ and $\rho\in(0,\frac12R)$, set $\theta=k^{2-p}$.
There exist  $\nu_o,\,\dl\in(0,1),\, \widetilde{\boldsymbol\gm}>1$ depending only on 
 the  $\mathfrak{data}$, such that if
\begin{equation*}
	|\{
	u(\cdot, t_o)\leqslant  k\}\cap  B_{\varrho}(x_o)|
	\leqslant
	\nu_o|B_{\varrho} |,
\end{equation*}
then either
\begin{equation*}%\label{Eq:DG:alt}
\widetilde{\boldsymbol\gm} 
{\rm Tail}[ u_-; \mathcal{Q}]
%\int^{T_2}_{T_1}\int_{\rn\setminus B_R(x_o)}\frac{ \big(u - \boldsymbol \mu^{\pm}\big)_{\pm}^{p-1}}{|x|^{N+sp}}\,\dx\dt
>k
\end{equation*}
or
\begin{equation*}
	u\geqslant \tfrac18k
	\quad
	\mbox{a.e.~in  $B_{\frac12\rho}(x_o) \times(t_o+\tfrac12 \dl\theta\rho^{sp},t_o+\dl\theta\rho^{sp}]$,}
\end{equation*}
provided that  $B_{2\rho}(x_o) \times(t_o,t_o+\dl\theta\rho^{sp}]\subset\mathcal{Q}$.
\end{lemma}

\begin{proof}
After a translation, we may assume $(x_o,t_o)=(0,0)$; we will slightly abuse the symbol $\mathcal{Q}=B_R\times(T_1,T_2]$ to denote the shifted cylinder. We 
 introduce for $n\in \nn_0$,
\begin{align*}
	\left\{
	\begin{array}{c}
\dsty k_n= 2^{-1}k +2^{-(n+1)} k ,
\\[5pt]
\dsty \rho_n=\rho+2^{-n}\rho, \quad\widetilde{\rho}_n=\tfrac12(\rho_n+\rho_{n+1}),\\[5pt]
\dsty \widehat{\rho}_n=\tfrac34\rho_n+ \tfrac14\rho_{n+1},\quad \overline{\rho}_n=\tfrac14\rho_n+ \tfrac34\rho_{n+1},\\[5pt]
\dsty B_n=B_{\rho_n}, \quad \widetilde{B}_n=B_{\widetilde{\rho}_n},\quad \widehat{B}_n=B_{\widehat\rho_n},\quad \overline{B}_n=B_{\overline{\rho}_n},\\[5pt]
\dsty Q_n=B_n\times(0,\dl\theta\rho^{sp}],\quad\widetilde{Q}_n=\widetilde{B}_n\times(0,\dl\theta\rho^{sp}],\\[5pt]
\widehat{Q}_n=\widehat{B}_n\times(0,\dl\theta\rho^{sp}],\quad\overline{Q}_n=\overline{B}_n\times(0,\dl\theta\rho^{sp}].
\end{array}
	\right.
\end{align*}
Observe that $Q_{n+1}\subset\overline{Q}_n\subset\widetilde{Q}_n\subset\widehat{Q}_n\subset Q_n$ and their height is fixed to be $\dl\theta\rho^{sp}$, where $\dl$ is to be chosen.
The cutoff function $\z\in C_0^1(B_n,[0,1])$ satisfies $\z=1$ on $\widetilde{B}_n$ and $|\nabla \z|\leqslant 2^{n+4}/\rho$.
The energy estimate of Proposition~\ref{Prop:2:1} is used in $Q_n$ with $\z$ and
with 
$$
w_-(x,t):=(u(x,t)+\ell(t)-k_n)_-, \quad \ell(t):=\widetilde{\boldsymbol\gm}\int^t_{T_1}\int_{\rn\setminus B_R}\frac{ u_{-}^{p-1}(y,\tau)}{|y|^{N+sp}}\,\dy\d\tau,
$$
where $\widetilde{\boldsymbol\gm}>1$ is to be determined in terms of the $\mathfrak{data}$.
As a result, we have
\begin{align}\nonumber
	\max\bigg\{&\essup_{t\in[0,\dl \theta\rho^{sp}]} \int_{\widetilde{B}_n} w^2_{-}(x,t)\,\dx,\, \iiint_{\widetilde{\mathbb{Q}}_n}   \frac{|w_-(x,t) - w_-(y,t)|^p}{|x-y|^{N+sp}}\,\dx\dy\dt\bigg\} \\ \nonumber
	%&\qquad+\iint_{Q(R,S)} w_{\pm}(y,t)\z^p(y,t)\,\dy\dt \bigg(\int_{ B_R}  \frac{w^{p-1}_\mp(x,t)}{|x-y|^{N+sp}}\,\dx\bigg)
	 %+\int_{B_R}\z^p w^2_{\pm}(x,t)\,\dx\bigg|_{t_o-S}^{t_o}\\
	&\leqslant
	\boldsymbol\gm\iiint_{\mathbb{Q}_n} \max\{w^p_{-}(x,t), w^p_{-}(y,t)\} \frac{|\z(x) - \z(y)|^p}{|x-y|^{N+sp}}\,\dx\dy\dt\\\nonumber
	&\quad+\boldsymbol\gm\iint_{Q_n} \z^p w_{-}(x,t)\,\dx\dt \bigg[\essup_{\substack{x\in \widehat{B}_n}} \int_{\rn\setminus B_n}\frac{ w_{-}^{p-1}(y,t)}{|x-y|^{N+sp}}\,\dy\bigg]\\\label{Eq:energy-DG}
	&\quad-2\iint_{Q_n}  \ell^{\prime}(t) \z^pw_{-}(x,t)\,\dx\dt + \int_{B_n} w_{-}^2(x,0)\,\dx.
\end{align}
Here, the symbols $\widetilde{\mathbb{Q}}_n$ and $\mathbb{Q}_n$ should be self-explanatory.

Now, we deal with the right-hand side of \eqref{Eq:energy-DG}. 
The estimation of the first term  is standard, whereas the second and third terms need to be packed. This procedure also determines $\widetilde{\boldsymbol{\gm}}$ in the definition of $\ell(t)$, cf.~\cite[Lemma~3.1]{Liao-mod-24}.
 As a result, the first three terms on the right-hand side of \eqref{Eq:energy-DG} are bounded by
\[
 \boldsymbol\gm 2^{(N+sp)n} \frac{k^p}{\rho^{sp}} |A_n|.
\]
Note that there is no $\dl$-dependence in the above bound.
To estimate the last term in \eqref{Eq:energy-DG} observe that, since $\rho_n\in(\rho, 2\rho]$ and $k_n\leqslant k$, we have
\begin{align*}
|\{u(\cdot, 0)<k_n\}\cap B_n| &\leqslant |\{u(\cdot, 0)<k\}\cap B_{2\rho}|\\
&\leqslant \nu_o |B_{2\rho}| \leqslant 2^N\nu_o |B_n| =\frac{2^N\nu_o}{\dl\theta\rho^{sp}}|Q_n|,
\end{align*}
where $\nu_o$ is still to be selected.
Then, it is estimated by
\[
\int_{B_n}  w_{-}^2(x,0)\,\dx \leqslant k^2 |\{u(\cdot, 0)<k_n\}\cap B_n|\leqslant 2^N\frac{ k^{p}}{\dl \rho^{sp}}\nu_o|Q_n|.
\]
Collecting these estimates on the right-hand side of \eqref{Eq:energy-DG}, we arrive at
\begin{align}\nonumber\label{Eq:DG-energy:1}
\essup_{t\in[0,\dl\theta\rho^{sp}]}&\int_{\widetilde{B}_n} w^2_{-}(x,t)\,\dx+ \iiint_{\widetilde{\mathbb{Q}}_n}   \frac{|w_-(x,t) - w_-(y,t)|^p}{|x-y|^{N+sp}}\,\dx\dy\dt\\
&\leqslant \boldsymbol\gm 2^{(N+2p)n}\frac{k^p}{ \rho^{sp}}|A_n|   +2^N\frac{k^{p}}{\dl \rho^{sp}}\nu_o|Q_n|.
\end{align}

To proceed, we consider {\it two cases} by comparing the two terms on the rigt-hand side. In the {\it first case}, suppose
\begin{equation}\label{Eq:DG-alt:1}
|A_n|=|\{u <k_n\}\cap Q_n|\leqslant \frac{\nu_o}{\dl}|Q_n|\quad\text{for some}\> \> n\in  \nn_0.
\end{equation}
Noticing that $\rho_n\in(\rho, 2\rho]$ and $k_n\geqslant\frac12 k$, 
then \eqref{Eq:DG-alt:1} yields that
\begin{align}\label{Eq:meas-alt:1}\nonumber
%&\big|\big\{u <\tfrac14 k\big\}\cap (0,\dl\rho^{2s})+Q_{\rho}(\dl)\big|\\
 |\{u<\tfrac12 k\}\cap (0,\dl\theta\rho^{sp})+Q_{\rho}(\dl)| &\leqslant|\{u<k_n\}\cap Q_n| \\
&\leqslant\frac{\nu_o}{\dl}|Q_n|\leqslant 2^N\frac{\nu_o}{\dl} |Q_{\rho}(\dl\theta)|.
\end{align}
Recall our notation $(0,\dl\theta\rho^{sp})+Q_{\rho}(\dl\theta)=B_{\rho}\times(0,\dl\theta\rho^{sp}]$ in \eqref{Eq:meas-alt:1}.
Assuming that $\dl$ has been fixed for the moment, let us choose $\nu_o$ to be so small that 
\begin{equation}\label{Eq:nu:0}
2^N\frac{\nu_o}{\dl}\leqslant \nu\quad\implies\quad\nu_o\leqslant 2^{-(N+q_1)}\widetilde{\boldsymbol\gm}^{-1} \dl^{q_2+1}
\end{equation}
where we plugged in $\nu=(1-a)^{q_1}\widetilde{\boldsymbol\gm}^{-1}\dl^{q_2}$ from \eqref{Eq:nu-dep} with $a=\frac12$, while $\widetilde{\boldsymbol\gm}$, $q_1$ and $q_2$ depend only on the $\mathfrak{data}$, as determined in Lemma~\ref{Lm:DG:1}. Then   \eqref{Eq:meas-alt:1} and \eqref{Eq:nu:0} allow us to apply Lemma~\ref{Lm:DG:1} and conclude that
\begin{equation}\label{Eq:DG:pt-est:1}
u\geqslant\tfrac1{8}k\quad\text{ a.e. in }(0,\dl\theta\rho^{sp})+Q_{\frac12\rho}(\dl\theta),
\end{equation}
provided that $\ell(T_2)\leqslant\frac18k$.
This step does not require any quantitative knowledge of $\dl$, which is still to be specified.

In the {\it second case}, suppose
\begin{equation}\label{Eq:DG-alt:2}
|A_n|=|\{u <k_n\}\cap Q_n|> \frac{\nu_o}{\dl}|Q_n|\quad\text{for any}\>\> n\in  \nn_0.
\end{equation}
Then, the energy estimate \eqref{Eq:DG-energy:1} gives that for any $n\in  \nn_0$,
\begin{equation*}
\begin{aligned}
	\essup_{t\in[0,\dl\theta\rho^{sp}]}&\int_{\widetilde{B}_n} w^2_{-}(x,t)\,\dx+ \iiint_{\widetilde{\mathbb{Q}}_n}  \frac{|w_-(x,t) - w_-(y,t)|^p}{|x-y|^{N+sp}}\,\dx\dy\dt\\
	&\quad\leqslant
	 \boldsymbol\gm 2^{(N+2p)n}\frac{k^p}{ \rho^{sp}}|A_n| .
\end{aligned}
\end{equation*}
Departing from the above energy estimate, one can run DeGiorgi's iteration exactly as in \cite[Lemma~3.2]{Liao-cvpd-24} and \cite[Lemma~3.2]{Liao-mod-24}. As a result, we find $\dl$ depending only on the $\mathfrak{data}$, such that
\begin{equation}\label{Eq:DG:pt-est:2}
u \geqslant\tfrac14k \quad
	\mbox{a.e.~in $ B_{\frac12\rho} \times(0,\dl\theta\rho^{sp}]$,}
\end{equation}
provided that $\ell(T_2)\leqslant\frac14k$.
%This means that, if we impose \eqref{Eq:l:1}, then
%\begin{equation}\label{Eq:DG:pt-est:2}
%u \geqslant\tfrac14 k\quad
%	\mbox{a.e.~in $ B_{\frac12\rho} \times\big(0,\dl \rho^{2s}\big]$.}
%\end{equation}
Once $\dl$ is chosen, this also fixes $\nu_o$ from \eqref{Eq:nu:0} in terms of the $\mathfrak{data}$. Therefore, no matter which case of \eqref{Eq:DG-alt:1} and \eqref{Eq:DG-alt:2} occurs, one always obtains the desired pointwise estimate from \eqref{Eq:DG:pt-est:1} and \eqref{Eq:DG:pt-est:2}. 
\end{proof}

The next two results can be viewed as refinements of Lemma~\ref{Lm:meas-shrink:1}.  Comparing with Lemma~\ref{Lm:meas-shrink:1}, the following result encodes the measure theoretical information in a local integral; it generalizes \cite[Lemma~3.5]{Liao-DD-24} to an intrinsic version.

\begin{lemma}\label{Lm:meas-shrink:2}
 Let $u$ be a   local weak super-solution to \eqref{Eq:1:1}--\eqref{Eq:K} that is nonnegative in $\mathcal{Q}$. Let $c,\,k>0$ be  parameters.
There exists
 $\boldsymbol \gm>1$ depending only on the $\mathfrak{data}$, such that
   either 
  $$ 
{\rm Tail} [ u_-; \mathcal{Q} ]
> ck
  $$ 
  or
\begin{equation*}
	\essinf_{t\in[t_o-ck^{2-p}\rho^{2s}, t_o]}|\{
	u(\cdot, t)\leqslant  k \}\cap  B_{\rho}(x_o)|
	\leqslant \frac{ \boldsymbol\gm\max\{1,c^{-1}\} k^{p-1} }{[u^{p-1}]_{(x_o,t_o)+Q_\rho(ck^{2-p})} } |B_{\rho}|,
\end{equation*}
where the denominator on the right-hand side is the integral average of $u^{p-1}$ on $(x_o,t_o)+Q_{\rho}(ck^{2-p})$,
provided that $(x_o,t_o)+Q_{2\rho}(ck^{2-p})\subset \mathcal{Q}$.
\end{lemma}
%%%
\begin{proof}
Assume $(x_o,t_o)=(0,0)$ up to a translation and still use $\mathcal{Q}$ to denote the shifted reference cylinder. 
Let us  employ the energy estimate  of Proposition~\ref{Prop:2:1} in $ B_{2\rho}\times (-ck^{2-p}\rho^{sp}, 0]$ 
with the truncation
	$w_- =(u-k)_-$.
Introduce
a  cutoff function $\z\in C^1_0(B_{3\varrho/2},[0,1])$ such that $\z=1$ in $B_{\varrho}$ and $|\nabla\z|\leqslant4/\varrho$.
Then, we obtain  that
\begin{align*}
	&\iint_{Q_\rho(ck^{2-p})} w_{-}(y,t) \,\dy\dt \bigg[\int_{ \rn }  \frac{w^{p-1}_+(x,t)}{|x-y|^{N+sp}}\,\dx\bigg]\\
    &\leqslant\boldsymbol\gm\iiint_{\mathbb{Q}(2\rho,ck^{2-p}\rho^{sp})}\max\{w^p_{-}(x,t),\, w^p_{-}(y,t)\} \frac{|\z(x) - \z(y)|^p}{|x-y|^{N+sp}}\,\dx\dy\dt\\
	&\qquad+\boldsymbol\gm\iint_{Q(2\rho,ck^{2-p}\rho^{sp})} \z^p w_{-}(x,t)\bigg[\essup_{x\in B_{\frac32\rho}}\int_{\rn\setminus B_{2\rho}}\frac{ w_{-}^{p-1}(y,t)}{|x-y|^{N+sp}}\,\dy\bigg]\dx\dt \\
    &\qquad+\int_{B_{2\varrho} } w_-^2(x, -ck^{2-p}\rho^{sp}) \,\dx.
	% & \leqslant \int_{B_{2\varrho} } w_-^2(x, -ck^{2-p}\rho^{sp}) \,\dx +
	% \frac{\boldsymbol\gm}{\rho^{2s}} \int_{-\rho^{2s}}^{0}\int_{B_{2\rho}} w_{-}^2(x,t) \,\dx \dt\\
	% &\quad+\boldsymbol\gm\int_{-\rho^{2s}}^0\int_{B_{2\rho}}  w_{-}(x,t)\,\dx\dt \bigg(  \int_{\rn\setminus B_{\frac32\rho}}\frac{ w_{-}(y,t)}{|y|^{N+2s}}\,\dy\bigg).
\end{align*}
The various terms on the right-hand side are dealt with similarly to \cite[Lemma~3.4]{Liao-cvpd-24} and \cite[Lemma~3.4]{Liao-mod-24}.
Namely, the first term  is estimated by
\begin{align*}
\boldsymbol{\gm}\frac{k^p}{\rho^{sp}}|Q_\rho(ck^{2-p})|=\boldsymbol\gm ck^2 |B_\rho|.
\end{align*}
Note that putting $\rho$ or $2\rho$ in the above display does not make essential difference.
As for the second term, observe that for $|x|\leqslant \frac{3}{2}\rho$ and $|y|\geqslant2\rho$, we have $|x-y|\geqslant\frac14|y|$; hence, we estimate it by
\begin{align*}
4^{N+sp}&\iint_{Q(2\rho,ck^{2-p}\rho^{sp})}  w_{-}(x,t)\,\dx\dt \bigg[ \int_{\rn\setminus B_{\frac32\rho}}\frac{ w^{p-1}_{-}(y,t)}{|y|^{N+sp}}\,\dy\bigg]\\
&\leqslant 4^{N+sp} k |B_{2\rho} | \bigg[\int_{-ck^{2-p} \rho^{sp}}^0 \int_{\rn\setminus B_{\frac32\rho}}\frac{ w^{p-1}_{-}(y,t)}{|y|^{N+sp}}\,\dy\dt\bigg]\\
&\leqslant 4^{N+sp} k |B_{2\rho}| \bigg[\boldsymbol\gm  ck +\int_{T_1}^{T_2} \int_{\rn\setminus B_{R}}\frac{ u^{p-1}_{-}(y,t)}{|y|^{N+sp}}\,\dy\dt\bigg]\\
&\leqslant \boldsymbol\gm ck^2 |B_{2\rho}|=\boldsymbol\gm 2^N ck^2 |B_\rho|.
\end{align*}
To obtain the last line, we enforced the tail estimate
\[
\int^{T_2}_{T_1}\int_{\rn\setminus B_R}\frac{ u^{p-1}_{-}(y,t)}{|y|^{N+sp}}\,\dy\dt\leqslant c k.
\]
The third term is standard:
\[
\int_{B_{2\varrho} } w_-^2(x,- ck^{2-p}\rho^{sp}) \,\dx\leqslant k^2|B_{2\rho}|=  2^N k^2 |B_\rho|.
\]
Combining the above estimate we see that the energy estimate becomes
\begin{align}\label{Eq:shrink:0}\nonumber
	\iint_{Q_\rho(ck^{2-p})}& w_{-}(y,t) \,\dy\dt \bigg[\int_{ \rn }  \frac{w^{p-1}_+(x,t)}{|x-y|^{N+sp}}\,\dx\bigg]\\
	& \leqslant
\boldsymbol\gm \max\{1,c\} k^2 |B_\rho|
\end{align}
for some $\boldsymbol\gm$ depending only on the $\mathfrak{data}$.
%provided the tail estimate is enforced.

To estimate the left-hand side of \eqref{Eq:shrink:0}, we evaluate the global integral in the bracket only over $B_\rho$ and use that $|x-y|\leqslant 2\rho$ if $x,\,y\in B_\rho$.
As a result, we have that
\begin{align*}
&\iint_{Q_\rho (ck^{2-p})}w_{-}(y,t) \,\dy\dt \bigg[\int_{ B_{\rho}}  \frac{w^{p-1}_+(x,t) }{|x-y|^{N+sp}}\,\dx\bigg]\\
&\geqslant  \iint_{Q_\rho (ck^{2-p})}w_{-}(y,t) \,\dy\dt \bigg[\int_{ B_{\rho}}  \frac{w^{p-1}_+(x,t) }{(2\rho)^{N+sp}}\,\dx\bigg]\\
&\geqslant  \essinf_{t\in[-ck^{2-p}\rho^{sp}, 0]}\int_{B_\rho}w_{-}(y,t) \,\dy \bigg[\iint_{ Q_{\rho}(ck^{2-p})}  \frac{w_+^{p-1}(x,t) }{(2\rho)^{N+sp}}\,\dx\dt \bigg]\\
&\geqslant   \essinf_{t\in[-ck^{2-p}\rho^{sp}, 0]}\int_{B_\rho}w_{-}(y,t) \,\dy \bigg[\iint_{ Q_{\rho}(ck^{2-p})}\frac{u^{p-1}(x,t) - k^{p-1}}{(2\rho)^{N+sp}}\,\dx\dt\bigg]\\
&\geqslant  \essinf_{t\in[-ck^{2-p}\rho^{sp}, 0]} \int_{B_\rho}w_{-}(y,t) \,\dy \bigg[ \frac{ck^{2-p}}{\boldsymbol{\gm}} \biint_{Q_\rho(ck^{2-p})} u^{p-1}\,\dx\dt - \boldsymbol{\gm} ck \bigg]\\
&\geqslant  \frac{ck^{2-p}}{\boldsymbol{\gm}} \essinf_{t\in[-ck^{2-p}\rho^{sp}, 0]}\int_{B_\rho}w_{-}(y,t) \,\dy \cdot [u^{p-1}]_{Q_\rho(ck^{2-p})} - \boldsymbol{\gm} ck^2 |B_{\rho}|,
\end{align*}
 whereas the last integral can be estimated from below by integrating over a smaller set, i.e.,
\[
\essinf_{t\in[-ck^{2-p}\rho^{sp}, 0]}\int_{B_\rho}w_{-}(y,t) \,\dy \geqslant \tfrac12 k \essinf_{t\in[-ck^{2-p}\rho^{sp}, 0]}|\{u(\cdot, t)\leqslant\tfrac12 k\}\cap B_\rho|.
\]
Combining these estimates in \eqref{Eq:shrink:0}, we obtain the desired estimate upon adjusting constants.
\end{proof}
%%%

The last result generalizes \cite[Lemma~3.6]{Liao-DD-24} to an intrinsic version. As discussed in \S~\ref{S:method}, it plays a crucial role in balancing local and nonlocal contributions.
\begin{lemma}\label{Lm:meas-shrink:3}
 Let $u$ be a  local weak super-solution to \eqref{Eq:1:1}--\eqref{Eq:K} that is nonnegative in $\mathcal{Q}$.
 Let $c,\,k>0$ be parameters. 
There exists
 $\boldsymbol \gm>1$ depending only on the $\mathfrak{data}$,  such that
 either 
 $$   
{\rm Tail} [ u_-; \mathcal{Q} ]
>ck
 $$ 
 or
\begin{equation*}
\essinf_{t\in[t_o-ck^{2-p}\rho^{sp}, t_o]}|\{u(\cdot,t)\leqslant  k \}\cap B_\rho(x_o)|\leqslant  \frac{ \boldsymbol\gm \max\{1,c\}k}{{\rm Tail}[u_+; (x_o,t_o)+Q_{\rho}(ck^{2-p})]} |B_\rho|,
\end{equation*}
provided that $(x_o,t_o)+Q_{2\rho}(ck^{2-p})\subset \mathcal{Q}$.
\end{lemma}
%%%%%%
\begin{proof}
Assume $(x_o,t_o)=(0,0)$. The argument departs from \eqref{Eq:shrink:0}.
To estimate the left-hand side, we integrate instead over $\rn\setminus B_{\rho}$ and use the fact that when $|y|\leqslant \rho$ and $|x|\geqslant \rho$, one has $|x-y|\leqslant 2 |x|$. As a result, we have that
\begin{align*}
&\iint_{Q_\rho (ck^{2-p})}w_{-}(y,t) \,\dy\dt \bigg[\int_{\rn\setminus B_{\rho}}  \frac{w^{p-1}_+(x,t) }{|x-y|^{N+sp}}\,\dx\bigg]\\
&\geqslant  \iint_{Q_\rho (ck^{2-p})}w_{-}(y,t) \,\dy\dt \bigg[\int_{\rn\setminus B_{\rho}}  \frac{w^{p-1}_+(x,t) }{|2x|^{N+sp}}\,\dx\bigg]\\
&\geqslant  \essinf_{t\in[-ck^{2-p}\rho^{sp}, 0]}\int_{B_\rho}w_{-}(y,t) \,\dy \bigg[\int_{-ck^{2-p}\rho^{sp}}^0\int_{\rn\setminus B_{\rho}}  \frac{w^{p-1}_+(x,t) }{|2x|^{N+sp}}\,\dx\dt \bigg]\\
&\geqslant   \essinf_{t\in[-ck^{2-p}\rho^{sp}, 0]}\int_{B_\rho}w_{-}(y,t) \,\dy \bigg[\int_{-ck^{2-p}\rho^{sp}}^0\int_{\rn\setminus B_{\rho}}\frac{u^{p-1}_+(x,t) - k^{p-1}}{|2x|^{N+sp}}\,\dx\dt\bigg]\\
&\geqslant  \essinf_{t\in[-ck^{2-p}\rho^{sp}, 0]} \int_{B_\rho}w_{-}(y,t) \,\dy \bigg[\frac{1}{\boldsymbol{\gm}}{\rm Tail}[u_+, Q_{\rho}(ck^{2-p})] - \boldsymbol\gm ck \bigg]\\
&\geqslant  \frac{1}{\boldsymbol{\gm}} \essinf_{t\in[-ck^{2-p}\rho^{sp}, 0]}\int_{B_\rho}w_{-}(y,t) \,\dy \cdot {\rm Tail}[u_+, Q_{\rho}(ck^{2-p})] - \boldsymbol\gm c k^2 |B_{\rho}|.
\end{align*}
The proof is then concluded just like in Lemma~\ref{Lm:meas-shrink:1}.
\end{proof}

%%%%%%%%%%%%%%

\section{Expansion of positivity}\label{S:exp-pos}
In this section, we establish two results regarding expansion of positivity. The first result features a from-measure-to-pointwise estimate.
 We will continue to use the reference cylinder $\mathcal{Q}:=B_R(x_o)\times(T_1,T_2]\subset E_T$.
\begin{proposition}\label{Prop:expansion}
Let $u$ be a   local  weak super-solution to \eqref{Eq:1:1}--\eqref{Eq:K} that is nonnegative in  $\mathcal{Q}$, with $p>2$. 
Suppose for some $\al\in(0,1]$ and  $k>0$, we have
	\begin{equation*}
		|\{ u(\cdot, t_o) \geqslant k \}\cap B_{\varrho}(x_o) |
		\geqslant
		\al |B_\varrho |.
		%\quad
		%\mbox{ for all $t_o<t< t_o+\dl\boldsymbol\om^{q+1-p}\varrho^p$}.
	\end{equation*}
There exist constants $\dl\in(0,1)$ depending  on the $\mathfrak{data}$ and $\eta$ depending additionally on $\al$,  such that 
either
\begin{equation*}
{\rm Tail}[u_-; \mathcal{Q}]
\geqslant \eta k
\end{equation*}
 or
\begin{equation*}
	u\geqslant\eta  k
	\quad
	\mbox{a.e.~in $ B_{2\varrho}(x_o) \times( t_o+\tfrac34  \dl (\eta k)^{2-p}(4\varrho)^{sp},
	t_o+\dl (\eta k)^{2-p}(4\varrho)^{sp}],$}
\end{equation*}
with $\eta=\al^{1+(N+2p)(p-1)}/\boldsymbol{\gm}$, provided  
\[
B_{4\rho}(x_o)\times(t_o, t_o+ \dl(\eta k)^{2-p}(4\varrho)^{sp}]\subset \mathcal{Q}.
\]
%Moreover, $\eta=\al^{N+2p+1}/\boldsymbol{\gm}$ for some $\boldsymbol{\gm}>1$ depending on the data $\{s, p, N, C_o, C_1\}$.
\end{proposition}
%%%
\begin{proof}
%It suffices to deal with super-solutions and to  and $\boldsymbol\mu^{-}=0$ for simplicity.
 Assume $(x_o,t_o)=(0,0)$ up to a translation and still denote the shifted reference cylinder by $\mathcal{Q}$. Rewrite the measure theoretical information at the initial time $t_o=0$ in the larger ball $B_{4\rho}$ and replace $\al$ by $4^{-N}\al$. Afterwards,
 apply Lemma~\ref{Lm:meas-propogation} to obtain $\dl_o,\, \varep\in(0,1)$ depending only on the $\mathfrak{data}$ and $\al$, such that 
	\begin{equation}\label{Eq:exp-pos-meas:1}
	|\{
	 u(\cdot, t) \geqslant \varep k\} \cap B_{4\varrho} |
	\geqslant\frac{\al}2 4^{-N} |B_{4\varrho}|
	\quad\mbox{ for all $t\in(0, \dl_o k^{2-p}(4\varrho)^{sp}]$,}
\end{equation}
provided we enforce that
\[
\frac1{\dl_o}
{\rm Tail} [ u_-; \mathcal{Q}]
\leqslant k.
\]
The dependence is traced by 
\begin{align}\label{eps-dl-choice}
    \varep=\frac{\al}{4^{N+1}},\qquad \dl_o=\frac{\al^{N+p+1}}{\boldsymbol{\gm}}.
\end{align}
This measure estimate \eqref{Eq:exp-pos-meas:1} for each slice of the time interval in turn allows us to apply Lemma~\ref{Lm:meas-shrink:1} with parameters $(\al,\dl,\sig, k,\rho)$ replaced by $$(2^{-2N-1}\al,\dl_1,\sig_1, \varep k,4\rho);$$ we further require the parameters to satisfy 
\[
B_{4\varrho}\times (0, \dl_o k^{2-p}(4\varrho)^{sp}]\equiv B_{4\varrho}\times (0, \dl_1(\sig_1\varep k)^{2-p}(4\varrho)^{sp}].
\]
Comparing the time intervals on both sides of the above display, one finds that this is the case if  $\dl_1$ and $\sig_1$ fulfill the relation
\begin{align}\label{dl-1-sig-1}\nonumber
    \dl_1(\sig_1\varep)^{2-p}=\dl_o\quad&\Longleftrightarrow\quad \dl_1=\dl_o(\sig_1\varep)^{p-2}\\
    &\Longleftrightarrow\quad \dl_1=\frac{\al^{N+2p-1}}{\boldsymbol{\gm}}\sig_1^{p-2}
\end{align}
where the last expression exploits the definitions of $\varep$ and $\dl_o$ in \eqref{eps-dl-choice}, and $\boldsymbol{\gm}$ depends only on the $\mathfrak{data}$. Then, under this identification of cylinders, Lemma~\ref{Lm:meas-shrink:1} yields a measure estimate:
\begin{align}\label{Eq:measure-est}\nonumber
	|\{
	    u \leqslant \sig_1 \varep k \}&\cap B_{4\varrho}\times (0, \dl_o k^{2-p}(4\varrho)^{sp}]|\\ \nonumber
       &=|\{
	   u \leqslant \sig_1 \varep k \}\cap B_{4\varrho}\times (0, \dl_1(\sig_1\varep k)^{2-p}(4\varrho)^{sp}]|\\ 
	&\leqslant \boldsymbol\gm \frac{\sig_1^{p-1}}{\dl_1\al} |Q_{4\rho}(\dl_o k^{2-p})| = \boldsymbol{\gm}_*\frac{\sig_1}{\al^{N+2p}}|Q_{4\rho}(\dl_o k^{2-p})|,
\end{align}
where the last equality exploits $\dl_1$ identified in \eqref{dl-1-sig-1}, and $\boldsymbol{\gm}_*$ can be determined by the $\mathfrak{data}$ only, provided that
  $$ 
\frac1{\dl_1}  
{\rm Tail}[u_-; \mathcal{Q}]
\leqslant\sig_1\varep k,
  $$ 
and that the cylinder under consideration is included in $\mathcal{Q}$. Next, let $\nu_o$ be determined in Lemma~\ref{Lm:DG:2} in terms of the $\mathfrak{data}$, and choose $\sig_1$ to satisfy that
\begin{equation}\label{sigma-choice}
     \boldsymbol{\gm}_*\frac{\sig_1}{\al^{N+2p}}\leqslant \nu_o\quad\implies\quad \sig_1=\lm\al^{N+2p} \quad\text{where}\>\>\lm\in(0,\nu_o/ \boldsymbol{\gm}_*].
\end{equation} 
Note that in the above choice of $\sig_1$, the parameter $\lm$ can be made smaller at our will.
Once $\sig_1$ fulfills \eqref{sigma-choice}, from \eqref{Eq:measure-est} we can find an instant $t_*\in[0, \dl_o k^{2-p}(4\varrho)^{sp}]$ such that
\[
|\{
	   u(\cdot,t_*) \leqslant \sig_1 \varep k \}\cap B_{4\varrho} ]|\leqslant \nu_o |B_{4\varrho}|.
\]
Based on the above measure estimate, Lemma~\ref{Lm:DG:2} gives some $\dl$ and $\widetilde{\boldsymbol\gm}$  depending only on the $\mathfrak{data}$ and ensures that
\begin{align}\label{Eq:exp-pos:2}
    u\geqslant \tfrac18\sig_1\varep k
\end{align}
a.e. in
\[
    B_{2\rho}\times(t_*+\tfrac12\dl(\sig_1\varep k)^{2-p}(4\rho)^{sp}, t_*+\dl(\sig_1\varep k)^{2-p}(4\rho)^{sp}],
\]
provided that  the cylinder under consideration is included in $\mathcal{Q}$ and that
\begin{equation*}%\label{Eq:DG:alt}
\widetilde{\boldsymbol\gm} 
{\rm Tail}[u_-; \mathcal{Q}]\leqslant \sig_1\varep k.
\end{equation*}
The precise location of $t_*$ in $[0, \dl_o k^{2-p}(4\varrho)^{sp}]$ is unknown notwithstanding, the above pointwise estimate \eqref{Eq:exp-pos:2} always holds true in the cylinder
\[
B_{2\rho}\times( \tfrac34\dl(\sig_1\varep k)^{2-p}(4\rho)^{sp},  \dl(\sig_1\varep k)^{2-p}(4\rho)^{sp}].
\]
This is because we can adjust $\lm\in(0,\nu_o/ \boldsymbol{\gm}_*]$ in the definition \eqref{sigma-choice} of $\sig_1$, such that
no matter where  the instant $t_*$  is in $[0, \dl_o k^{2-p}(4\varrho)^{sp}]$, we can always conclude that
\begin{align*}
    &( \tfrac34\dl(\sig_1\varep k)^{2-p}(4\rho)^{sp},  \dl(\sig_1\varep k)^{2-p}(4\rho)^{sp}]\\
    &\qquad \subset(t_*+\tfrac12\dl(\sig_1\varep k)^{2-p}(4\rho)^{sp}, t_*+\dl(\sig_1\varep k)^{2-p}(4\rho)^{sp}].
\end{align*}
In fact, to get the above inclusion, it suffices to ensure
\begin{align*}
    \dl_o k^{2-p}(4\varrho)^{sp} + \tfrac12\dl(\sig_1\varep k)^{2-p}(4\rho)^{sp}\leqslant \tfrac34\dl(\sig_1\varep k)^{2-p}(4\rho)^{sp},
\end{align*}
that is,
\[
\dl_o\leqslant\tfrac14\dl(\sig_1\varep)^{2-p}.
\]
Note that $\varep,\,\dl_o\leqslant1$, and hence the above line is satisfied if we require
\[
\sig_1\leqslant(\tfrac{1}{4}\dl)^{\frac1{p-2}}.
\]
Therefore, taking into account the first smallness condition set in \eqref{sigma-choice}, we only need to pick  
\[
\lm=\min\bigg\{\Big( \frac{\dl}{4}\Big)^{\frac1{p-2}}, \,\frac{\nu_o}{\boldsymbol{\gm}_*} \bigg\}.
\]
In this way, the desired inclusion of time intervals is realized. The proof is concluded by taking all tail conditions into consideration and properly redefining parameters. 
\end{proof}

\begin{remark}\upshape
    The power-like dependence of $\eta=\eta(\al)$ is a crucial feature. This kind of result requires much more work for the $p$-parabolic type equations, cf.~\cite[Chapter~4]{DBGV-mono}.
\end{remark}
%%%%
Based on the first result, we formulate a corollary that examines how positivity grows from a tiny ball to a large one at later time.
\begin{corollary}\label{Cor:expansion}
Let $u$ be a   local weak super-solution to \eqref{Eq:1:1}--\eqref{Eq:K} that is nonnegative in $\mathcal{Q}$, with $p>2$.
Suppose for some constants  $\varep, \al  \in(0,1]$ and $k>0$, we have
	\begin{equation*}
		|\{ u(\cdot, t_o) \geqslant k \}\cap B_{\varep\varrho}(x_o) |
		\geqslant
		\al |B_{\varep\varrho} |.
		%\quad
		%\mbox{ for all $t_o<t< t_o+\dl\boldsymbol\om^{q+1-p}\varrho^p$}.
	\end{equation*}
There exist constants $ \be>1$, $\dl \in(0,1)$ depending only on the $\mathfrak{data}$ and $ \eta\in(0,1)$ depending  additionally on $\al$, such that 
%either $$|\boldsymbol \mu^{\pm}|>\xi M$$ or
%enforcing $C_2^{\frac1{p-1}}\rho>\xi M$, we have 
either
\begin{equation*}
	 {\rm Tail}[ u_{-}; \mathcal{Q}] >\eta \varep^{\be} k
\end{equation*}
 or
\begin{equation*}
	u\geqslant\eta \varep^{\be} k
	\quad
	\mbox{a.e.~in $ B_{2\varrho}(x_o) \times( t_o+\tfrac12 \dl (\eta \varep^{\be} k)^{2-p}\varrho^{sp},
	t_o+\dl (\eta\varep^{\be} k)^{2-p}\varrho^{sp}],$}
\end{equation*}
provided  
\[
B_{32\rho}(x_o)\times(t_o, t_o+\dl (\eta \varep^{\be} k)^{2-p}\varrho^{sp}]\subset \mathcal{Q}.
\]
Moreover, we have $\dl=2^{5+2sp-3p}\dl_o$ where $\dl_o$ is determined in Lemma \ref{Lm:DG:initial:1} and $\eta$ has the same dependence on $\al$ as in Proposition \ref{Prop:expansion}.
\end{corollary}
%%%%
\begin{proof}
Assume $(x_o,t_o)=(0,0)$ after a translation and keep denoting the translated $\mathcal{Q}$ by the same symbol. Starting from the given measure theoretical information at $t_o=0$, we apply Proposition~\ref{Prop:expansion} to obtain $\eta,\,\dl\in(0,1)$, depending on the $\mathfrak{data}$ and $\al$, such that, after enforcing
\[
	{\rm Tail}[ u_{-}; \mathcal{Q}]\leqslant \eta  k,
\]
we have
\begin{equation*}
	u\geqslant\eta k
	%\quad
	%\mbox{a.e.~in $ B_{2\varep\varrho}  \times\underbrace{\big( \tfrac34 \dl (\eta M)^{2-p}(2^2\varep\varrho)^{sp},
	 %\dl (\eta M)^{2-p}(2^2\varep\varrho)^{sp}\big]}_{:=(\frac34 t_1,t_1]}.$}
\end{equation*}
a.e.~in $$ B_{2\varep\varrho}  \times\underbrace{( \tfrac34 \dl (\eta k)^{2-p}(2^2\varep\varrho)^{sp},
	 \dl (\eta k)^{2-p}(2^2\varep\varrho)^{sp}]}_{:=(\frac34 t_1,t_1]}.$$
 Using this pointwise estimate and applying again Proposition~\ref{Prop:expansion} with $\al=1$, we obtain $\bar{\eta}$ and $\bar\dl$ in $(0,1)$, depending only on the $\mathfrak{data}$, such that, after enforcing
\[
	  {\rm Tail}[ u_{-}; \mathcal{Q}] \leqslant \eta \bar{\eta} k,
\]
we have
\begin{equation*}
	u\geqslant\eta \bar{\eta} k
\end{equation*}
a.e.~in $$ B_{2^2\varep\varrho}  \times(  \tfrac34 t_1+\tfrac34 \bar\dl (\eta  \bar{\eta} k)^{2-p}(2^3\varep\varrho)^{sp},
	 t_1+\bar\dl (\eta \bar{\eta} k)^{2-p}(2^3\varep\varrho)^{sp}].$$
Repeating this argument $n$ times, we have, after enforcing
\[
	 {\rm Tail}[ u_{-}; \mathcal{Q}] \leqslant \eta  \bar{\eta}^n k,
\]
the pointwise estimate
\begin{equation*}
	u\geqslant\eta \bar{\eta}^n k
\end{equation*}
a.e.~in 
$$
B_{2^{n+1}\varep\varrho}  \times\Big( \tfrac34 t_1+ \sum_{i=1}^{n} \tfrac34 \bar\dl (\eta  \bar{\eta}^i k)^{2-p}(2^{i+2}\varep\varrho)^{sp}, t_1+ \sum_{i=1}^{n}\bar\dl (\eta \bar{\eta}^i k)^{2-p}(2^{i+2}\varep\varrho)^{sp}\Big].
$$
This procedure requires the set inclusion:
$$
B_{2^{n+2}\varep\varrho}  \times\Big( 0, t_1+ \sum_{i=1}^{n}\bar\dl (\eta \bar{\eta}^i k)^{2-p}(2^{i+2}\varep\varrho)^{sp}\Big] \subset \mathcal{Q}.
$$

Now, we choose the integer $n$ to satisfy $4\leqslant 2^n\varep<8$. In this way, the above pointwise estimate yields that
\begin{equation}\label{est-t-bar}
	u(\cdot,\bar{t})\geqslant\eta  \bar{\eta}^4  \varep^{\be} k =:\widetilde{\eta}k
	\quad
	\mbox{a.e.~in $ B_{8 \varrho}$,}
\end{equation}
for $$\be :=-\log_2\bar\eta$$ and for some $\bar{t}\geqslant0 $,
once we enforce
\begin{equation*}
	 {\rm Tail}[ u_{-}; \mathcal{Q}] \leqslant %\eta 2^{\log_2\bar{\eta}}\varep^{\be'} M =
     \widetilde{\eta}k.
\end{equation*}
The instant $\bar{t}$ has an upper bound which can be estimated  as follows:
\begin{align*}
    \bar{t}&\leqslant t_1 +\sum_{i=1}^{n}\bar\dl (\eta \bar{\eta}^i k)^{2-p}(2^{i+2}\varep\varrho)^{sp}\\
    &\leqslant (\eta k)^{2-p}(4\varep\rho)^{sp} \bigg[1 +    \sum_{i=1}^{n} (2^{sp}\bar\eta^{2-p})^i\bigg]\\
    &= (\eta k)^{2-p}(4\varep\rho)^{sp}\bigg[1 +  \frac{(2^{sp}\bar\eta^{2-p})^{n+1}-2^{sp}\bar\eta^{2-p}}{2^{sp}\bar\eta^{2-p} -1}\bigg]\\
    &\leqslant (\eta k)^{2-p}(4\varep\rho)^{sp}\bigg[1 +  \Big(\frac{16}{\varep}\Big)^{\log_2(2^{sp}\bar{\eta}^{2-p})}\bigg]\\
    &\leqslant 4^{ p}(\eta \bar{\eta}^4 k)^{2-p}(4\varep\rho)^{sp}  \varep^{-\log_2(2^{sp}\bar{\eta}^{2-p})}\\
    &=4^{ p}(\eta \bar{\eta}^4\varep^{\be} k)^{2-p}(4\rho)^{sp}=: t_* 
\end{align*}
where we used the definition of $\be$ to combine the $\varep$-terms in the last line.
The upper bound $t_*$ of $\bar t$ will help us locate the quantitative lower bound of $u$ next.
Indeed, using \eqref{est-t-bar}, we apply Lemma \ref{Lm:DG:initial:1} to get
\begin{align}\label{lower-bd-at-t-bar}
    u\geqslant\tfrac14\xi\widetilde{\eta}k	\quad
	\mbox{a.e.~in $ B_{4\varrho}\times (\bar{t},\bar{t}+\dl_o(\xi\widetilde{\eta}k)^{2-p}(4\rho)^{sp}]$,}
\end{align}
for $\dl_o=\dl_o(\mathfrak{data})$ and for some free parameter $\xi\in(0,1)$ to be fixed, provided that the above cylinder is included in $\mathcal{Q}$ and that
\[
\widetilde{\boldsymbol{\gm}}{\rm Tail}[ u_{-}; \mathcal{Q}] \leqslant \xi\widetilde{\eta}k.
\]
Let us choose $\xi$ to satisfy
\begin{align*}
(t_*,2t_*]&\equiv (4^{ p}(\eta \bar{\eta}^4\varep^{\be} k)^{2-p}(4\rho)^{sp}, 2\cdot 4^{ p}(\eta \bar{\eta}^4\varep^{\be} k)^{2-p}(4\rho)^{sp}]\\
&\subset(\bar{t},\bar{t}+\dl_o(\xi\widetilde{\eta}k)^{2-p}(4\rho)^{sp}]. 
\end{align*}
To this end, we only need to ensure
\begin{align*}
    \dl_o(\xi\widetilde{\eta}k)^{2-p}(4\rho)^{sp}\geqslant  2\cdot 4^{ p}(\eta \bar{\eta} ^4\varep^{\be} k)^{2-p}(4\rho)^{sp} .
\end{align*}
Upon using the definition of $\widetilde{\eta}$ in \eqref{est-t-bar}  in the above display, we find the choice 
$$
\xi:=\Big(\frac{\dl_o}{8^p}\Big)^{\frac{1}{p-2}}.
$$
Therefore, we obtain from \eqref{lower-bd-at-t-bar} that
\begin{align*}
    u\geqslant\tfrac14\xi\eta\bar{\eta}^4\varep^{\be}k	\quad
	\mbox{a.e.~in $ B_{4\varrho}\times (t_*,2t_*]$,}
\end{align*}
or equivalently, a.e.~in
$$
B_{4\varrho}\times(4^{ p}(\eta \bar{\eta}^4\varep^{\be} k)^{2-p}(4\rho)^{sp}, 2\cdot4^{ p}(\eta \bar{\eta}^4\varep^{\be} k)^{2-p}(4\rho)^{sp}],
$$
provided that
$$
B_{32\varrho}\times(0, 2\cdot4^{ p}(\eta \bar{\eta}^4\varep^{\be} k)^{2-p}(4\rho)^{sp}] \subset \mathcal{Q}.
$$
This is what we aim to show if we choose $\dl:=2^{5+2sp}\xi^{p-2}=2^{5+2sp-3p}\dl_o$ and redefine $\frac14\xi\eta\bar{\eta}^4$ as $\eta$.
\end{proof}
%%%%%%%%%%%%%%%%%%%%%%%%%%%%%%%%%%%%%%%
 
%%%%%%%%%%%
\section{Proof of Harnack type estimates}\label{S:Thm:1-proof}
This section is devoted to the proof of Theorem~\ref{Thm:1}. 
Under a change of variables, let us introduce a new function
\begin{equation}\label{Eq:u-to-v}
    v(x,t)
    :=
    \frac{u (x_o+\rho x, t_o+[u(x_o,t_o)]^{2-p}\rho^{sp} t )}{u(x_o,t_o)}\quad\text{in}\>\> E_T^*, % \widetilde{Q}_8=B_8\times(-8^{sp}, 8^{sp}].
\end{equation}
where $E_T^*$ denotes the domain that is transformed from $E_T$ under the change of variables.
This new function satisfies $v(0,0)=1$ and solves
\begin{equation*} 
\pl_t v + \mathscr{L}^* v=0\quad\text{weakly in}\>\> E_T^*,
\end{equation*}
%where $E_T=E\times(0,T]$ 
where the nonlocal operator $\mathscr{L}^*$ is defined by
\begin{equation*}%\label{Eq:1:2}
\mathscr{L}^*v(x,t):={\rm p.v.}\int_{\rn} \mathsf{K}^*(x,y,t) |v(x,t) - v(y,t) |^{p-2}  (v(x,t) - v(y,t) )\,\dy
\end{equation*}
and
\begin{equation*}
    \mathsf{K}^*(x,y,t):=\rho^{N+sp}\mathsf{K}(x_o+\rho x, x_o+\rho y,t_o+[u(x_o,t_o)]^{2-p}\rho^{sp} t)
\end{equation*}
for $(x,t),\,(y,t) \in E_T^*$. It is not hard to see that $\mathsf{K}^*(x,y,t)$ fulfills the same bounds as in \eqref{Eq:K}. The notion of solution is the same as in Definition~\ref{Def:sol}.
Thus $v$ is granted with the properties shown in previous sections within the domain $E_T^*$.

With this preparation at hand, we take on the proof in four steps. The first step turns out to be the most difficult one.
The reader is advised to take $v_-\equiv0$ on first reading.

%Note that since $v\geqslant0$ in $E_T^*$, $v_-\equiv0$ in $E_T^*$,
\subsection{The forward Harnack type estimate}
In this section, we aim to show that
there exist constants $\boldsymbol{\gm}_{\mathsf{h}}>1$ and $\kp,\,\sig\in(0,1)$ depending on the $\mathfrak{data}$, such that if  
\begin{equation}\label{proof:condi:1}
    B_{100}\times (-2\cdot8^{sp}\sig^{2-p} , 2\cdot 8^{sp} \sig^{2-p}  ]\subset E^*_T,
\end{equation} 
 %where $\theta=\sig^{2-p}$ 
 and if
\begin{equation}\label{proof:condi:2}
    \boldsymbol{\gm}_{\mathsf{h}}\mathrm{Tail}[v_-; B_{100}\times(-2\cdot 8^{sp}\sig^{2-p},2\cdot8^{sp}\sig^{2-p}]]\leqslant \sig,
\end{equation}
then we have the forward Harnack estimate:
    \begin{equation}\label{proof:concl:1}
         \inf_{B_1}u (\cdot, \kp\sig^{2-p} )\geqslant \sig.
    \end{equation} 
Note that since $v_-\equiv0$ in $E_T^*$ by assumption, its tail remains the same if we replace $B_{100}$ in \eqref{proof:condi:2} by a smaller ball.

Let $\be=\be(\mathfrak{data})$ be selected in Corollary \ref{Cor:expansion}. Consider the following two functions with a variable $r\in(0,1)$:
\begin{align*}
    \mathsf{M}_r:=\sup_{Q_r} v,\qquad \mathsf{N}_r:=(1-r)^{-\be}.
\end{align*}
Observe that they are nondecreasing along $r$ and $\mathsf{M}_o = \mathsf{N}_o=1$. Let us denote $r_*$ the largest root of the equation $\mathsf{M}_r=\mathsf{N}_r$, $r\in(0,1)$. Since $u$ is continuous, there is some $(y,\tau)\in \overline{Q}_{r_*}$ such that
\begin{align}\label{v-max}
    v(y,\tau)=\mathsf{M}_{r_*}=\mathsf{N}_{r_*}=(1-r_*)^{-\be}.
\end{align}
In addition, observe that
\[
(y,\tau)+Q_{\frac12(1-r_*)}\subset Q_{\frac12(1+r_*)}\subset Q_1.
\]
Since $\frac12(1+r_*)>r_*$, by definition of $r_*$, we have 
\[
\mathsf{M}_{\frac12(1+r_*)}\leqslant \mathsf{N}_{\frac12(1+r_*)}
\]
which joint with the above set inclusion yields that
\begin{align}\label{Eq:v-bound}
    \sup_{(y,\tau)+Q_{\frac12(1-r_*)}} v \leqslant \sup_{ Q_{\frac12(1+r_*)}} v = \mathsf{M}_{\frac12(1+r_*)}\leqslant \mathsf{N}_{\frac12(1+r_*)} = 2^\be (1-r_*)^{-\be}.
\end{align}
Before proceeding to the next step, let us pick some parameters
\begin{equation}\label{Eq:parameters}
\left\{
    \begin{array}{cc}
            \displaystyle\boldsymbol{\mu}^+ =\boldsymbol{\om} = 2^\be (1-r_*)^{-\be}, \quad R= \tfrac14(1-r_*),  \\ [5pt]
        \displaystyle   \xi= 1- \frac{1}{2^{\be+1}} ,\quad a= \frac{1-\displaystyle\frac32  \frac{1}{2^{\be+1}}}{1- \displaystyle\frac{1}{2^{\be+1}}} .
    \end{array}\right.
\end{equation}
Note that while $\be$ has been fixed quantitatively in terms of the $\mathfrak{data}$, we only know $r_*$ qualitatively. Therefore, we aim to eliminate $r_*$ from the final estimate.

Our next step consists in finding some $\nu,\,\sig,\, \eta\in(0,1)$ depending only on the $\mathfrak{data}$, such that
\begin{align} \label{claim-measure}\nonumber
&\text{either}\>\>\mathrm{Tail}[v_-; B_8\times(-2\cdot 8^{sp}\sig^{2-p},2\cdot8^{sp}\sig^{2-p}]]\geqslant \eta\\
 &\text{or}\>\>   |\{v(\cdot, \bar{\tau})\geqslant \eta (1-r_*)^{-\be}\}\cap B_{2R}(y)|\geqslant \nu|B_{2R}|%\quad\text{for some}\>\bar{\tau}\in \big[\tau- (\sig\xi\boldsymbol{\om})^{2-p} R^{sp}, \tau\big].
\end{align}
holds true for some 
\begin{equation}\label{Eq:tau-bar}
    \bar{\tau}\in  [\tau- (\sig\xi\boldsymbol{\om})^{2-p} (2R)^{sp}, \tau ],
\end{equation}
 provided that
\begin{equation*}
    B_{36}\times(-2\cdot8^{sp}\sig^{2-p},2\cdot8^{sp}\sig^{2-p}]\subset E_T^*.
\end{equation*}
To this end, we consider {\it two cases}. In the {\it first case}, let us suppose
\begin{align}\label{case-1}
    %\widetilde{\boldsymbol\gm}\int^\tau_{\tau- (\xi\boldsymbol{\om})^{2-p} R^{sp}}\int_{\rn\setminus B_R(y)}\frac{ v_{+}^{p-1}(x,t)}{|x-y|^{N+sp}}\,\dx\dt
    \widetilde{\boldsymbol\gm}\mathrm{Tail}[v_+; (y,\tau)+Q_{2R}((\xi\boldsymbol{\om})^{2-p})] \leqslant \xi\boldsymbol\om
\end{align}
holds true. Recall that in \eqref{case-1}, $(y,\tau)$ is a point in $\overline{Q}_{r_*}$ such that \eqref{v-max} holds; the parameters $R$,  $\xi$ and $\boldsymbol{\om}$ are defined in \eqref{Eq:parameters}; whereas $\widetilde{\boldsymbol{\gm}}$ is determined in Lemma~\ref{Lm:DG:1} in terms of the $\mathfrak{data}$. Moreover, we choose
$$\widetilde\nu=\frac{(1-a)^{q_1}}{\widetilde{\boldsymbol\gm}}$$
according to \eqref{Eq:nu-dep} in Lemma~\ref{Lm:DG:1} with $\dl=1$. Note that $\widetilde\nu$ depends only on the $\mathfrak{data}$ as $q_1$, $a$ and $\widetilde{\boldsymbol{\gm}}$ do.
Then, we assert that
\begin{equation}\label{Eq:mearure>nu}
    |\{v\geqslant \boldsymbol{\mu}^+-\xi\boldsymbol{\om}\}\cap (y,\tau)+Q_R((\xi\boldsymbol{\om})^{2-p})|>\widetilde\nu |Q_R((\xi\boldsymbol{\om})^{2-p})|,
\end{equation}
from which the measure theoretical estimate in \eqref{claim-measure}  follows with $\eta=\frac12$ and $\nu=2^{-N}\widetilde{\nu}$, noting that $\boldsymbol{\mu}^+-\xi\boldsymbol{\om}=\tfrac12 (1-r_*)^{-\be}$.
Let us suppose \eqref{Eq:mearure>nu} were false. Then, we appeal to Lemma~\ref{Lm:DG:1} with $\dl=1$, the parameters  defined in \eqref{Eq:parameters}, and the cylinders
\begin{equation*}
\left\{
    \begin{array}{cc}
        (x_o,t_o)+Q_\rho (\theta)=(y,\tau)+Q_R ((\xi\boldsymbol{\om})^{2-p} ), \\[5pt]
        \mathcal{Q}=(y,\tau)+Q_{2R}((\xi\boldsymbol{\om})^{2-p})\subset E_T^*.
    \end{array}\right.
\end{equation*}
The last set inclusion is ensured by \eqref{proof:condi:1}.
Under the conditions \eqref{Eq:v-bound} and \eqref{case-1}, the failure of \eqref{Eq:mearure>nu} would yield that
$$
v(y,\tau)\leqslant \boldsymbol{\mu}^+-a\xi\boldsymbol{\om}=\tfrac34 (1-r_*)^{-\be},
$$
a contradiction to \eqref{v-max}.

In the {\it second case}, we instead assume the reverse of \eqref{case-1}:
 \begin{align*}
    %\widetilde{\boldsymbol\gm}\int^\tau_{\tau- (\xi\boldsymbol{\om})^{2-p} R^{sp}}\int_{\rn\setminus B_R(y)}\frac{ v_{+}^{p-1}(x,t)}{|x-y|^{N+sp}}\,\dx\dt
    \widetilde{\boldsymbol\gm}\mathrm{Tail}[v_+; (y,\tau)+Q_{2R}((\xi\boldsymbol{\om})^{2-p})]> \xi\boldsymbol\om
\end{align*}
holds true. Using this assumption, we apply Lemma~\ref{Lm:meas-shrink:3} with the cylinders
\begin{equation*}
\left\{
    \begin{array}{cc}
      \mathcal{Q}=B_8\times(-2\cdot8^{sp}\sig^{2-p},2\cdot8^{sp}\sig^{2-p}],\\[5pt]
    (x_o,t_o)+Q_{\rho}(k^{2-p})=(y,\tau)+Q_{2R}(k^{2-p}),\\[5pt]
    (y,\tau)+Q_{4R}(k^{2-p})\subset \mathcal{Q},
\end{array}\right.
\end{equation*}
for $k=\sig\xi\boldsymbol{\om}$ and $\sig\in(0,1)$ to be selected. Note that the last set inclusion comes from the fact that $\tau\in(0,1)$, $k^{2-p}(4R)^{sp}\leqslant\sig^{2-p}$, and then the time intervals of the cylinders are nesting as 
\begin{equation}\label{Eq:nest-interval}
    1+k^{2-p}(4R)^{sp}\leqslant 2\cdot 8^{sp}\sig^{2-p}.
\end{equation}
In this way,  we find from Lemma~\ref{Lm:meas-shrink:3} that
\begin{align*}\nonumber
&\text{either}\>\>\mathrm{Tail}[v_-; B_8\times(-2\cdot8^{sp}\sig^{2-p},2\cdot8^{sp}\sig^{2-p}]]\geqslant\sig\xi\boldsymbol{\om}  \\
 &\text{or}\>\>   
    \essinf_{t\in[\tau- (\sig\xi\boldsymbol{\om})^{2-p} (2R)^{sp}, \tau]} |\{v(\cdot, t)\leqslant \sig\xi\boldsymbol{\om}\}\cap B_{2R}(y)|\leqslant \boldsymbol{\gm}\widetilde{\boldsymbol{\gm}}\sig|B_{2R}|.
\end{align*}
Let us set the first smallness requirement to $\sig$ by
\begin{equation}\label{Eq:sig:1}
    \sig\leqslant \frac{1}{2\boldsymbol{\gm}\widetilde{\boldsymbol{\gm}}}.
\end{equation}
We will set more smallness requirements to $\sig$ in the course of the proof.
However, assume momentarily that $\sig$ has been fixed. Then, the claim \eqref{claim-measure} with $\eta=2^\be\sig\xi$ and $\nu=\frac12$ follows from the above either-or estimate. Therefore, the desired alternative \eqref{claim-measure} holds in either of the two cases if we choose 
$$
\eta:=\min\bigg\{\frac12,\,2^\be\sig\xi\bigg\}, \quad \displaystyle\nu:=\min\bigg\{\frac12,\,\frac{(1-a)^{q_1}}{2^N\widetilde{\boldsymbol\gm}}\bigg\}.
$$
It is apparent that the above choices of parameters depend only on the $\mathfrak{data}$, once $\sig$ is chosen in terms of the $\mathfrak{data}$.

With the measure theoretical estimate \eqref{claim-measure} at hand, we now apply Corollary \ref{Cor:expansion} to $v$ with $\mathcal{Q}=B_{66}\times(-2\cdot8^{sp}\sig^{2-p},2\cdot8^{sp}\sig^{2-p}]$, at $(x_o,t_o)=(y,\bar\tau)$ and replacing $(\varep,\rho,\al,k)$ by $(2R,1,\nu,\eta(1-r_*)^{-\be})$. This application gives some $\dl$ and $\bar\eta$ in $(0,1)$ depending only on the $\mathfrak{data}$, such that
\begin{align}\label{Eq:v-lower-bd-1} \nonumber
 \text{either}\>\>\mathrm{Tail}[v_-; B_{66}\times(-2\cdot8^{sp}\sig^{2-p},2\cdot8^{sp}\sig^{2-p}]]\geqslant\underbrace{\bar{\eta}\eta(1-r_*)^{-\be} (2R)^\be}_{=:\widetilde{\eta}}  \\
  \text{or}\>\>  
    v\geqslant \widetilde{\eta}\quad\text{in}\> \>B_4(y)\times (\bar{\tau}+\tfrac12\dl\widetilde{\eta}^{2-p},\bar{\tau}+ \dl\widetilde{\eta}^{2-p}],
\end{align}
provided that 
\begin{equation}\label{Eq:set_incl_1}
  B_{64}(y)\times (\bar{\tau}, \bar{\tau}+\dl\widetilde{\eta}^{2-p}]\subset  B_{66}\times(-2\cdot8^{sp}\sig^{2-p},2\cdot8^{sp}\sig^{2-p}].
\end{equation}
Recalling that $R=\frac14(1-r_*)$, we find that the above-defined $\widetilde{\eta}$ is determined only by the $\mathfrak{data}$ and independent of $r_*$. To verify \eqref{Eq:set_incl_1}, it suffices to consider the time intervals. Given the possible range of $\bar\tau$ in \eqref{Eq:tau-bar}, we need
$$
1+k^{2-p}(2R)^{sp}\leqslant 2\cdot 8^{sp}\sig^{2-p}\quad\text{and}\quad \dl\widetilde{\eta}^{2-p}\leqslant 2\cdot 8^{sp}\sig^{2-p}.
$$
The first inequality is verified in \eqref{Eq:nest-interval}, whereas the second one gives the second smallness requirement of $\sig$ by
\begin{equation}\label{Eq:sig:2}
    \sig\leqslant\widetilde{\eta} .
\end{equation}
Note that the lower bound of $v$ in \eqref{Eq:v-lower-bd-1} can be taken in $B_2$ instead of $B_4(y)$ as $y\in B_1$.
Based on \eqref{Eq:v-lower-bd-1}, we further apply Lemma \ref{Lm:DG:initial:1} to $v$ with $\mathcal{Q}=B_{36}\times(-2\cdot8^{sp}\sig^{2-p},2\cdot8^{sp}\sig^{2-p}]$, $B_\rho(x_o)=B_2$, $k=\lm\widetilde{\eta}$ for  a free parameter $\lm\in(0,1)$ to be fixed, and $t_o$ ranging over the time interval in \eqref{Eq:v-lower-bd-1},  to conclude that
\begin{align}\label{Eq:v-lower-bd-2}\nonumber
&\text{either}\>\>\widetilde{\boldsymbol{\gm}}\mathrm{Tail}[v_-; B_{36}\times(-2\cdot8^{sp}\sig^{2-p},2\cdot8^{sp}\sig^{2-p}]]\geqslant\lm \widetilde{\eta}\\
 &\text{or}\>\>  
        v\geqslant \tfrac14 \lm \widetilde{\eta}\quad\text{in}\>\> B_1\times (\bar{\tau}+ \dl\widetilde{\eta}^{2-p},\bar{\tau}+ \dl\widetilde{\eta}^{2-p}+\dl_o(\lm\widetilde{\eta})^{2-p}]
\end{align}
for $\widetilde{\boldsymbol{\gm}}$ and $\dl_o$ depending only on the $\mathfrak{data}$,  provided that
\begin{align}\label{Eq:set-incl_2}\nonumber
    B_4\times &(\bar{\tau}+ \dl\widetilde{\eta}^{2-p},\bar{\tau}+ \dl\widetilde{\eta}^{2-p}+\dl_o(\lm\widetilde{\eta})^{2-p}]\\
    &\subset  B_{36}\times(-2\cdot8^{sp}\sig^{2-p},2\cdot8^{sp}\sig^{2-p}].
\end{align} 
Let us first select $\lm$ so small that
$$
 (\tfrac12\dl_o(\lm\widetilde{\eta})^{2-p}, \dl_o(\lm\widetilde{\eta})^{2-p}]\subset(\bar{\tau}+ \dl\widetilde{\eta}^{2-p},\bar{\tau}+ \dl\widetilde{\eta}^{2-p}+\dl_o(\lm\widetilde{\eta})^{2-p}];
$$
since $\bar{\tau}\leqslant0$, to get the above inclusion, it suffices to ensure that %\textcolor{red}{It seems that $\lm$ depends only on p, N.}
\begin{equation}\label{Eq:lm:0}
    \tfrac12\dl_o(\lm\widetilde{\eta})^{2-p}\geqslant \dl\widetilde{\eta}^{2-p}\quad\implies\quad \lm\leqslant\min\bigg\{1,\, \Big(\frac{\dl_o}{2\dl}\Big)^\frac{1}{p-2} \bigg\}.
\end{equation}
Then, to fulfill \eqref{Eq:set-incl_2}, it suffices to compare the right end of the time intervals in \eqref{Eq:set-incl_2}, that is, we require that
$$
\dl\widetilde{\eta}^{2-p}+\dl_o(\lm\widetilde{\eta})^{2-p}\leqslant 8^{sp}\sig^{2-p},
$$
which gives the third smallness requirement of $\sig$ by
\begin{equation*}%\label{Eq:sig:3}
    \sig\leqslant \frac{\lm\widetilde{\eta}}{2^\frac{1}{p-2}}.
\end{equation*}
This condition is automatically stronger than \eqref{Eq:sig:2}, and it will be stronger than \eqref{Eq:sig:1} if we impose that
\begin{equation}\label{Eq:lm:1}
    \lm\leqslant \frac{1}{2\boldsymbol{\gm}\widetilde{\boldsymbol{\gm}}}.
\end{equation}
With the restrictions of $\lm$ in \eqref{Eq:lm:0} and \eqref{Eq:lm:1}, we finally choose
\begin{equation}\label{Eq:sig-choice-final}
     \sig := \min\bigg\{\frac14,\, \frac{1}{2^\frac{1}{p-2}}\bigg\}\lm\widetilde{\eta}.
\end{equation}
Taking into consideration the tail alternatives in \eqref{claim-measure}, \eqref{Eq:v-lower-bd-1} and \eqref{Eq:v-lower-bd-2}, we conclude that if the set inclusion \eqref{proof:condi:1} holds and if
\begin{equation*}%\label{proof:condi:2}
    %\boldsymbol{\gm}_{\mathsf{h}}
    \frac{\widetilde{\boldsymbol{\gm}}}{\lm\widetilde{\eta}}\mathrm{Tail}[v_-; B_{36}\times(-2\cdot 8^{sp}\sig^{2-p},2\cdot8^{sp}\sig^{2-p}]]\leqslant 1,
\end{equation*}
then we have
\[
\inf_{B_1} v(\cdot, \dl_o(\lm\widetilde{\eta})^{2-p})\geqslant\tfrac14\lm\widetilde{\eta}.
\]
This is exactly \eqref{proof:condi:2} and \eqref{proof:concl:1}, if we define
\begin{align*}
    \kappa:= \dl_o\min\bigg\{\frac1{4^{p-2}},\, \frac{1}{2}\bigg\},\qquad\boldsymbol{\gm}_h:=\widetilde{\boldsymbol{\gm}} \min\bigg\{\frac14,\, \frac{1}{2^\frac{1}{p-2}}\bigg\}.
\end{align*}
We also remark that $\sig$ defined in \eqref{Eq:sig-choice-final} can be taken smaller by adjusting $\lm$ according to \eqref{Eq:lm:0} and \eqref{Eq:lm:1}.
Hence, the right-hand side inequality of \eqref{Thm:conclusion:1} is proven.

\subsection{The backward Harnack type estimate}
Let us keep using the transformed function $v$ defined in \eqref{Eq:u-to-v}. Fix the constants $\kappa$, $\sig$ and $\boldsymbol{\gm}_{\mathsf{h}}$ selected in the forward estimate. In this section, we aim to show the backward estimate:
\begin{align}\label{backward-claim}
   \sup_{\rho\in(0,1)} \sup_{B_\rho}v(\cdot, -(1+\varep)^{2-p}\kappa\rho^{sp})\leqslant \frac{1+\varep}{\sig}
\end{align}
holds for any $\varep\in(0,1)$,  provided that \eqref{proof:condi:1} and \eqref{proof:condi:2} are satisfied.  Then, we would finish the proof of the backward estimate by letting $\varep\to0$ in \eqref{backward-claim}.

Let us suppose to the contrary that \eqref{backward-claim} is false. Since $v$ is continuous and $v(0,0)=1$, there must be some $\rho_*\in(0,1]$ and some $y\in\overline{B}_{\rho_*}$, such that 
%where the supremum in \eqref{backward-claim}  is attained, i.e.
\begin{equation}\label{v-max-attained}
    v(y, -(1+\varep)^{2-p}\kappa\rho_*^{sp})=\frac{1+\varep}{\sig}.
\end{equation}
For simplicity, let us denote the time slice in \eqref{v-max-attained} by $\tau:=-(1+\varep)^{2-p}\kappa\rho_*^{sp}$, so that \eqref{v-max-attained} reads $v(y, \tau)=(1+\varep)/\sig$.  
Applying the forward estimate that has been proven in the previous section now at $(y,\tau)$, we acquire that
\begin{equation}\label{forward-est-at-max}
    \sig v(y, \tau)\leqslant  \inf_{B_{\rho_*}(y)} v(\cdot, \tau+\kappa\theta\rho_*^{sp})
\end{equation}
 where $\theta:=[\sig v(y, \tau)]^{2-p}$,
provided that
\begin{equation} \label{cylinder-inclusion}
    B_{66\rho_*}(y)\times (\tau-2\theta(8\rho_*)^{sp}, \tau+ 2\theta (8\rho_*)^{sp} ]\subset E_T^*
\end{equation} 
 and that
\begin{equation} \label{tail-control}
    \boldsymbol{\gm}_{\mathsf{h}}\int_{\tau-2\theta(8\rho_*)^{sp}}^{\tau+2\theta(8\rho_*)^{sp}}\int_{\rn\setminus B_{66\rho_*}(y)}\frac{v_-^{p-1}(x,t)}{|x-y|^{N+sp}}\,\dx\dt\leqslant \sig v(y, \tau).
\end{equation}
Observe that the left-hand side of \eqref{forward-est-at-max} is equal to $1+\varep$ due to \eqref{v-max-attained} while the right-hand side is bounded by $1$ as $\tau+\kappa\theta\rho_*^{sp}=0$ by definition of $\tau$ and $0\in B_{\rho_*}(y)$, and consequently the infimum in \eqref{forward-est-at-max} is bounded by $v(0,0)=1$. Therefore, a contradiction appears in \eqref{forward-est-at-max}, provided that we enforce \eqref{cylinder-inclusion} and \eqref{tail-control}. Using $\theta\leqslant 2^{2-p}$, $\tau\leqslant 1$ and $y\in B_{1}$, we see that \eqref{proof:condi:1} is  more stringent than \eqref{cylinder-inclusion}.
% \begin{equation} \label{cylinder-inclu-1}
%     B_{38}\times (-4\cdot 8^{sp}, 4\cdot 8^{sp} ]\subset E_T^*.
% \end{equation} 
Moreover, since $v_-\equiv0$ in $E_T^*$, we may estimate the integral in \eqref{tail-control} by
\begin{align*}
    \int_{-4\cdot 8^{sp}}^{4\cdot 8^{sp}}\int_{\rn\setminus B_{4}}\frac{v_-^{p-1}(x,t)}{|x-y|^{N+sp}}\,\dx\dt\leqslant
    \Big(\frac43\Big)^{N+sp}\int_{-4\cdot 8^{sp}}^{4\cdot 8^{sp}}\int_{\rn\setminus B_{4}}\frac{v_-^{p-1}(x,t)}{|x|^{N+sp}}\,\dx\dt
\end{align*}
where we also used the fact $|x-y|\geqslant\frac34|x|$ due to $|y|\leqslant 1$ and $|x|\geqslant4$, whence we see that \eqref{proof:condi:2} places a more stringent condition than \eqref{tail-control} if we replace $\boldsymbol{\gm}_{\mathsf{h}}$ by $(4/3)^{N+sp}\boldsymbol{\gm}_{\mathsf{h}}$.
Therefore, we have established \eqref{backward-claim} under the conditions \eqref{proof:condi:1} and \eqref{proof:condi:2}.

\subsection{The representation type estimate I}
Let us keep using $v$ defined in \eqref{Eq:u-to-v}. According to the backward estimate proven in the last section, we have
\begin{align}\label{sup-v-sig}
    \sup_{B_1\times (-2^{sp}\kappa, -\kappa]} v\leqslant\sig^{-1},
\end{align}
provided that \eqref{proof:condi:1} and \eqref{proof:condi:2} are verified. We apply Lemma~\ref{Lm:meas-shrink:3} with $k=\sig^{-1}$, $c=\sig^{2-p}\kappa(2^{sp}-1)$, and cylinders
\begin{equation*}
\left\{
    \begin{array}{cc}
        (x_o,t_o)+Q_{\rho}(ck^{2-p})= B_1\times(-2^{sp}\kappa,-\kappa], \\[5pt]
        \mathcal{Q}=B_{100}\times (-2\cdot8^{sp}\sig^{2-p} , 2\cdot 8^{sp} \sig^{2-p}  ].
    \end{array}\right.
\end{equation*}
Note that by \eqref{sup-v-sig} we have
$$
\essinf_{t\in(-2^{sp}\kappa,-\kappa]}|\{v(\cdot,t)\leqslant k= \sig^{-1} \}\cap B_1|=|B_1|,
$$
whence  Lemma~\ref{Lm:meas-shrink:3} yields that
\begin{align}\label{tail-sig}
    {\rm Tail}[v_+; B_1\times(-2^{sp}\kappa,-\kappa]]\leqslant \frac{\boldsymbol{\gm}}{\sig}
\end{align}
provided that  \eqref{proof:condi:1} and \eqref{proof:condi:2} are verified.
Likewise, an application of Lemma~\ref{Lm:meas-shrink:2} gives that
\begin{align}\label{mean-sig}
    [v^{p-1}]_{B_1\times(-2^{sp}\kappa,-\kappa]}\leqslant \frac{\boldsymbol{\gm}}{\kappa\sig}
\end{align}
provided that  \eqref{proof:condi:1} and \eqref{proof:condi:2} are verified.
Let us combine \eqref{tail-sig} and \eqref{mean-sig} to obtain that
\begin{equation}\label{Eq:backward:0}
    %\frac{1}{\boldsymbol{\gm}_{\mathsf{H}}}\,
    \int^{-\kappa}_{-2^{sp}\kappa}\int_{\rn}\frac{v_+^{p-1}(x,t)}{ (1+|x|)^{N+sp}}\,\dx\dt \leqslant \frac{\boldsymbol{\gm}^{\prime}_{\mathsf{h}}}{\kappa\sig}
\end{equation}
The desired estimate follows from \eqref{Eq:backward:0} if we revert to $u$.

\subsection{The representation type estimate II}
%The proof hinges upon the energy estimate in Proposition~\ref{Prop:2:1}.
The essence is an $L^\infty$-estimate of solutions based on the energy estimate and DeGiorgi's iteration; we adapt the scheme presented in \cite{Liao-DD-24} for linear equations.

First suppose $Q(R,\theta)\subset E_T$, and let $\rho\in(0,R)$.
Take $(x_o,t_o)=(0,0)$ for simplicity. 
For $n\in\nn_0$, $\lm\in(0,1)$, and $k>0$, introduce 
\begin{align*}
	\left\{
	\begin{array}{c}
 k_n= k - 2^{-n} k, 
\\[5pt]
 \rho_n=\lm\rho + 2^{-n}(1-\lm)\rho,\quad  \theta_n=\lm \theta + 2^{-n}(1-\lm)\theta ,\\[5pt]
 \widetilde{\rho}_n=\frac12(\rho_n+\rho_{n+1}), \quad \widehat{\rho}_n=\frac34\rho_n+ \frac14\rho_{n+1},\quad \overline{\rho}_n=\frac14\rho_n+ \frac34\rho_{n+1},\\[5pt]
 \widetilde{\theta}_n=\frac12(\theta_n+\theta_{n+1}), \quad \widehat{\theta}_n=\frac34\theta_n+ \frac14\theta_{n+1},\quad \overline{\theta}_n=\frac14\theta_n+ \frac34\theta_{n+1},\\[5pt]
  B_n=B_{\rho_n}, \quad \widetilde{B}_n=B_{\widetilde{\rho}_n},\quad \widehat{B}_n=B_{\widehat\rho_n},\quad \overline{B}_n=B_{\overline{\rho}_n},\\[5pt]
\dsty Q_n=B_n\times(-\theta_n,0),\quad\widetilde{Q}_n=\widetilde{B}_n\times(-\widetilde{\theta}_n,0),\\[5pt]
\widehat{Q}_n=\widehat{B}_n\times(-\widehat{\theta}_n,0),\quad\overline{Q}_n=\overline{B}_n\times(-\overline{\theta}_n,0).
\end{array}
	\right.
\end{align*}
Observe that $Q_o=Q(\rho,\theta)$, $Q_{\infty}=Q(\lm\rho,\lm\theta)$, and $Q_{n+1}\subset\overline{Q}_n\subset\widetilde{Q}_n\subset\widehat{Q}_n\subset Q_n$.
Let $\z\in C_0^1 (\widehat{Q}_n,[0,1])$ be such that $\z=1$ in $\widetilde{Q}_{n}$, and moreover,
\begin{equation*}
|\nabla \z|\leqslant \frac{2^{n+4}}{(1-\lm)\rho}\quad\text{ and }\quad |\pl_t \z|\leqslant \frac{2^{n+4}}{(1-\lm)\theta}.
\end{equation*}
The energy estimate of Proposition~\ref{Prop:2:1} is used in $Q_n$ with $\z$ and
with 
$$
w_+(x,t):= (u(x,t)-\ell(t)-k_{n+1} )_+, \quad \ell(t):=\widetilde{\boldsymbol\gm}\int^t_{-\theta}\int_{\rn\setminus B_R}\frac{ u^{p-1}_{+}(y,\tau)}{|y|^{N+2s}}\,\dy\d\tau,
$$
where $\widetilde{\boldsymbol\gm}>0$ is to be determined.
As a result, we have
\begin{align}\nonumber
	\sup_{t\in[-\widetilde{\theta}_n,0]}&\int_{\widetilde{B}_n} w^2_{+}(x,t)\,\dx+ \iiint_{\widetilde{\mathbb{Q}}_n}   \frac{|w_+(x,t) - w_+(y,t)|^p}{|x-y|^{N+sp}}\,\dx\dy\dt\\\nonumber
	&\quad\leqslant
	\boldsymbol\gm\iiint_{\mathbb{Q}_n} \max\{w^p_{+}(x,t), w^p_{+}(y,t)\} \frac{|\z(x,t) - \z(y,t)|^p}{|x-y|^{N+sp}}\,\dx\dy\dt\\\nonumber
	&\qquad+\boldsymbol\gm\iint_{Q_n} \z^p w_{+}(x,t)\,\dx\dt \bigg[\sup_{\substack{x\in \widehat{B}_n}} \int_{\rn\setminus B_n}\frac{ w^{p-1}_{+}(y,t)}{|x-y|^{N+sp}}\,\dy\bigg]\\\label{Eq:energy-DG-}
	&\qquad-2\iint_{Q_n}  \ell^{\prime}(t) \z^p w_{+}(x,t)\,\dx\dt + \iint_{Q_n} |\pl_t\z^p|w_{+}^2(x,t)\,\dx\dt.
\end{align}
The meaning of $\mathbb{Q}_n$ and $\widetilde{\mathbb{Q}}_n$ in \eqref{Eq:energy-DG-} is clear from the context.
Estimating the first term on the right-hand side of \eqref{Eq:energy-DG-} is standard:
\begin{align*}
\iiint_{\mathbb{Q}_n} &\max\{w^p_{+}(x,t), w^p_{+}(y,t)\} \frac{|\z(x,t) - \z(y,t)|^p}{|x-y|^{N+sp}}\,\dx\dy\dt\\
&\leqslant  \frac{2^{np+4p+1}}{(1-\lm)^p\rho^{p}}\iiint_{\mathbb{Q}_n}  \frac{w^p_{+}(x,t)}{|x-y|^{N+p(s-1)}}\,\dx\dy\dt\\
&\leqslant   \frac{\boldsymbol \gm 2^{np}}{(1-\lm)^p\rho^{sp}} \iint_{Q_n} w^p_{+}(x,t)\,\dx\dt.
\end{align*}
The last term is also standard, namely,
\[
\iint_{Q_n} |\pl_t\z^p|w_{+}^2(x,t)\,\dx\dt \leqslant\frac{2^{n+4}}{(1-\lm)\theta } \iint_{Q_n} w_{+}^2(x,t)\,\dx\dt.
\]
The second term and the third, negative term need to be packed.  To this end, observe that when $|y|\geqslant \rho_n$ and $|x|\leqslant \widehat\rho_n$, there holds
\[
\frac{|y-x|}{|y|}\geqslant1-\frac{\widehat\rho_n}{ \rho_n}=\frac14\Big(\frac{\rho_n-\rho_{n+1}}{\rho_n}\Big)\geqslant\frac{1-\lm}{2^{n+3}};
\]
when $|y|\geqslant R$ and $|x|\leqslant \rho$, there holds
\[
\frac{|y-x|}{|y|}\geqslant1-\frac{|x|}{|y|}\geqslant 1-\frac{\rho}{ R}. 
\]
Consequently, 
we estimate the second term as
\begin{align}\nonumber
\iint_{Q_n}&  \z^p w_{+}(x,t)\,\dx\dt \bigg[\sup_{\substack{x\in\widehat{B}_n}} \int_{\rn\setminus B_n}\frac{ w^{p-1}_{+}(y,t)}{|x-y|^{N+sp}}\,\dy\bigg]\\ \nonumber
&=\iint_{Q_n}  \z^p w_{+}(x,t)\,\dx\dt \\ \nonumber
&\qquad\cdot \sup_{\substack{x\in\widehat{B}_n}}\bigg[\int_{B_R\setminus B_n}\frac{ w^{p-1}_{+}(y,t)}{|x-y|^{N+sp}}\,\dy+\int_{\rn\setminus B_R}\frac{ w^{p-1}_{+}(y,t)}{|x-y|^{N+sp}}\,\dy\bigg]\\ \nonumber
&\leqslant   \frac{\boldsymbol\gm 2^{(N+sp)n}}{(1-\lm)^{N+sp}} \iint_{Q_n}  \z^p w_{+}(x,t)\,\dx\dt \bigg[ \int_{B_R\setminus B_n}\frac{ w^{p-1}_{+}(y,t)}{|y|^{N+sp}}\,\dy\bigg]\\ \nonumber
&\qquad+  \Big(\frac{R}{R-\rho}\Big)^{N+sp}   \iint_{Q_n}  \z^p w_{+}(x,t)\,\dx\dt \bigg[ \int_{\rn\setminus B_R}\frac{ w^{p-1}_{+}(y,t)}{|y|^{N+sp}}\,\dy\bigg]\\ \nonumber
&\leqslant  \frac{\boldsymbol\gm 2^{(N+sp)n}}{[\lm (1-\lm)]^{N+sp}\rho^{sp}}\Big(\frac{R}{\rho}\Big)^{N}\Big[\sup_{Q(R,\theta)}u\Big]^{p-1}   \iint_{Q_n}  \z^2 w_{+}(x,t)\,\dx\dt \\ \nonumber
&\qquad+\Big(\frac{R}{R-\rho}\Big)^{N+sp}   \iint_{Q_n}  \z^p w_{+}(x,t)\,\dx\dt  \bigg[ \int_{\rn\setminus B_R}\frac{ u^{p-1}_{+}(y,t)}{|y|^{N+sp}}\,\dy\bigg]. 
\end{align}
The last term in the above estimate will cancel with  the third, negative term on the right-hand side of the energy estimate \eqref{Eq:energy-DG-},
if we choose 
$$
2\widetilde{\boldsymbol\gm}=\Big(\frac{R}{R-\rho}\Big)^{N+sp}$$
which implies that
$$ \ell(t)=\frac{1}{2}\Big(\frac{R}{R-\rho}\Big)^{N+sp}\int^t_{-\theta}\int_{\rn\setminus B_R}\frac{ u^{p-1}_{+}(y,\tau)}{|y|^{N+sp}}\,\dy\d\tau.
$$ As a result of this choice, the second and the third terms in \eqref{Eq:energy-DG-} together are bounded by
\[
 \frac{\boldsymbol\gm 2^{(N+sp)n}\mathbf{I}}{[\lm(1-\lm)]^{N+sp}\rho^{sp}} \iint_{Q_n}  \z^p w_{+}(x,t)\,\dx\dt
 \]
where
 \[
 \mathbf{I}:=\Big(\frac{R}{\rho}\Big)^{N}\Big[\sup_{Q(R,\theta)}u\Big]^{p-1}.
\]
Before proceeding, let us observe the fact that
\begin{equation}\label{Eq:obs}
|A_n|:=|\{u(x,t) -\ell(t)\geqslant k_{n+1}\}\cap Q_n | \leqslant\frac{2^{p(n+1)}}{k^p} \iint_{Q_n}   \widetilde{w}^p_{+}(x,t)\,\dx\dt,
\end{equation}
where we denoted
\[
\widetilde{w}_+(x,t):= (u(x,t)-\ell(t)-k_{n} )_+.
\]
As a result of \eqref{Eq:obs}, noticing also that $w_+\leqslant \widetilde{w}_+$, we have
\begin{align*}
&\iint_{Q_n}  \z^p w_{+}(x,t)\,\dx\dt \\
&\leqslant \bigg[\iint_{Q_n}  w^p_{+}(x,t)\,\dx\dt \bigg]^{\frac1p} |\{u(x,t) -\ell(t)\geqslant k_{n+1}\}\cap Q_n |^{1-\frac1p}\\
&\leqslant \frac{2^{(n+1)(p-1)}}{k^{p-1}} \iint_{Q_n}   \widetilde{w}^p_{+}(x,t)\,\dx\dt.
\end{align*}
Likewise,
\begin{align*}
&\iint_{Q_n}   w^2_{+}(x,t)\,\dx\dt \\
&\leqslant \bigg[\iint_{Q_n}  w^p_{+}(x,t)\,\dx\dt \bigg]^{\frac2p} |\{u(x,t) -\ell(t)\geqslant k_{n+1}\}\cap Q_n |^{1-\frac2p}\\
&\leqslant \frac{2^{(n+1)(p-2)}}{k^{p-2}} \iint_{Q_n}   \widetilde{w}^p_{+}(x,t)\,\dx\dt.
\end{align*}
Collecting these estimates on the right-hand side of \eqref{Eq:energy-DG-} and using $w_+\leqslant \widetilde{w}_+$, we arrive at
\begin{align}\label{Eq:cacciopoli}\nonumber
\sup_{t\in[-\widetilde\theta_n,0]}&\int_{\widetilde{B}_n} w_+^2\,\dx +
\iiint_{\widetilde{\mathbb{Q}}_n}   \frac{|w_+(x,t) - w_+(y,t)|^p}{|x-y|^{N+sp}}\,\dx\dy\dt\\ \nonumber
&\leqslant \frac{\boldsymbol\gm 2^{(N+p)n}}{(1-\lm)^{N+p}}\bigg(\frac{k^{2-p}}{\theta } +\frac1{\rho^{sp}} +\frac1{\lm^{N+sp}\rho^{sp}}\frac{\mathbf{I}}{k^{p-1}}\bigg)  \iint_{Q_n}   \widetilde{w}^p_{+}\,\dx\dt\\
&\leqslant  \frac{\boldsymbol\gm 2^{(N+p)n}}{(1-\lm)^{N+p}\dl\lm^{N+sp}} \frac{1}{\rho^{sp}} \iint_{Q_n}   \widetilde{w}^p_{+}\,\dx\dt.
\end{align}
To obtain the last line, we enforced for a parameter $\dl\in(0,1)$ that
\begin{equation}\label{Eq:k-M}
k^{p-1}\geqslant \dl \mathbf{I}\quad\text{and}\quad k\geqslant \Big(\frac{\rho^{sp}}{\theta}\Big)^\frac{1}{p-2}.
\end{equation}

In order to properly use \eqref{Eq:cacciopoli}, consider  $\phi\in C_0^1 (\overline{Q}_n,[0,1])$, which  
 satisfies $\z=1$ in $Q_{n+1}$ and  $|\nabla\phi|\leqslant 2^{n+4}/\rho$.
 An application of H\"older's inequality, Sobolev's embedding
(cf.~\cite[Proposition~A.3]{Liao-cvpd-24} with $d=2^{-n-4}$), and \eqref{Eq:cacciopoli}  gives that 
\begin{align}\label{Eq:apply-sobolev}\nonumber
	&\iint_{Q_{n+1}} w_+^p\,\dx\dt
	\leqslant 
	\iint_{\widetilde{Q}_n}(\phi w_+)^p\,\dx\dt\\ \nonumber
	& \leqslant
	\bigg[\iint_{\widetilde{Q}_n}(\phi w_+)^{p\kappa }
	\,\dx\dt\bigg]^{\frac{1}{\kappa }}|A_n|^{1-\frac{1}{\kappa }}\\ \nonumber
	& \leqslant\boldsymbol\gm
	\bigg[\rho^{sp} \iiint_{\widetilde{\mathbb{Q}}_n} \frac{|\phi w_+(x,t) - \phi w_+(y,t)|^p}{|x-y|^{N+sp}}\,\dx\dy\dt+\iint_{\widetilde{Q}_n}\frac{(\phi w_+)^p}{d^{N+sp}}\,
	\dx\dt\bigg]^{\frac{1}{\kappa }}\\
	&  \qquad\qquad
	\cdot\bigg[\sup_{t\in[-\widetilde\theta_n,0]}
	\bint_{\widetilde{B}_n}(\phi w_+)^2\,\dx\bigg]^{\frac{\kappa_*-1}{\kappa_* \kappa }}
	 |A_n|^{1-\frac{1}{\kappa }}
\end{align}
where $|A_n|$ is defined in \eqref{Eq:obs} and $$\kappa:=1+\frac{2(\kappa_*-1)}{p\kappa_*}$$ with
\begin{equation*}
\kappa_*:=\left\{
\begin{array}{cl}
\frac{N}{N-sp} \quad&\text{if}\quad sp<N,\\[5pt]
\text{any number}\in (1,\infty) \quad &\text{if} \quad sp \geqslant N.
\end{array}\right.
\end{equation*} 
% \cite[Proposition~A.3]{Liao-cvpd-24} for the precise definition of $\kappa_*=\kappa_*(N,s)$.
To estimate the fractional norm in \eqref{Eq:apply-sobolev}, we use the triangle inequality
\begin{align*}
&|w_+\phi(x,t) - w_+\phi(y,t)|^p \\
&\quad\leqslant 2^{p-1}   |w_+(x,t) - w_+(y,t)  |^p \phi^p(x,t) 
 + 2^{p-1}  w^p_+(y,t)   |\phi(x,t) - \phi(y,t) |^p,
\end{align*}
such that
\begin{align*}
\rho^{sp}&\iiint_{\widetilde{\mathbb{Q}}_n}   \frac{|w_+\phi(x,t) - w_+\phi(y,t)|^p}{|x-y|^{N+sp}}\,\dx\dy\dt\\
&\leqslant 2^{p-1} \rho^{sp}\iiint_{\widetilde{\mathbb{Q}}_n}  \frac{|w_+(x,t) - w_+(y,t)|^p}{|x-y|^{N+sp}}\,\dx\dy\dt\\
&\quad + 2^{p-1} \rho^{sp}\iiint_{\widetilde{\mathbb{Q}}_n}  \frac{w^p_+(y,t)|\phi(x,t) - \phi(y,t)|^p}{|x-y|^{N+sp}}\,\dx\dy\dt\\
&\leqslant 2^{p-1} \rho^{sp}\iiint_{\widetilde{\mathbb{Q}}_n} \frac{|w_+(x,t) - w_+(y,t)|^p}{|x-y|^{N+sp}}\,\dx\dy\dt\\
&\quad + \boldsymbol\gm 2^{2n} \iint_{\widetilde{Q}_n}  w^p_+(y,t) \,\dy\dt\\
&\leqslant  \frac{\boldsymbol\gm 2^{(N+p)n}}{(1-\lm)^{N+p} \dl\lm^{N+sp}}  \iint_{Q_n}   \widetilde{w}^p_{+}(x,t)\,\dx\dt.
\end{align*}
Here, to obtain the last line we applied \eqref{Eq:cacciopoli}, and also used that $w_+\leqslant \widetilde{w}_+$ and $\widetilde{Q}_n\subset Q_n$. Observe the term containing $d$ in \eqref{Eq:apply-sobolev} can be estimated by the same quantity as in the last display, and  the term with $\sup_t$ in \eqref{Eq:apply-sobolev} is also estimated by this quantity according to \eqref{Eq:cacciopoli}, whereas the term $|A_n|$ in \eqref{Eq:apply-sobolev} is estimated by \eqref{Eq:obs}.
Therefore, we derive from \eqref{Eq:apply-sobolev} that
\begin{align*}
&\iint_{Q_{n+1}}  w_+^p\,\dx\dt\\
&\leqslant 
	\frac{\dsty\boldsymbol\gm\boldsymbol b^n  }{[(1-\lm)^{N+p}\dl\lm^{N+sp}]^{\frac{2\kappa_*-1}{\kappa_*\kappa }}} \frac{k^{p(\frac1\kappa -1)}}{\rho^{(N+sp)\frac{\kappa_* -1}{\kappa_*\kappa }} }  
	\bigg[\iint_{Q_{n}} \widetilde{w}_+^p\,\dx\dt\bigg]^{1+\frac{\kappa_* -1}{\kappa_*\kappa }}
\end{align*}
for some $\boldsymbol b=\boldsymbol b(N)>1$.
At this stage, we denote
\[
Y_n:=\frac{1}{k^p}\biint_{Q_{n}} (u(x,t)-\ell(t) - k_{n})_+^p\,\dx\dt,
\]
recall the definition of $\kappa$, and rewrite the previous recursive estimate as
\begin{align*}
    Y_{n+1} &\leqslant \frac{\dsty\boldsymbol\gm\boldsymbol b^n  }{[(1-\lm)^{N+p}\dl\lm^{N+sp}]^{\frac{2\kappa_*-1}{\kappa_*\kappa }}}\Big(\frac{\theta}{k^{p-2}\rho^{sp}}\Big)^{\frac{\kappa_* -1}{\kappa_*\kappa }}
Y_n^{1+\frac{\kappa_* -1}{\kappa_*\kappa }} \\
& \leqslant \frac{\dsty\boldsymbol\gm\boldsymbol b^n  }{[(1-\lm)^{N+p}\dl\lm^{N+sp}]^{\frac{2\kappa_*-1}{\kappa_*\kappa }}} 
Y_n^{1+\frac{\kappa_* -1}{\kappa_*\kappa }},
\end{align*}
where we used the second of \eqref{Eq:k-M}.
Hence, by the fast geometric convergence, cf. \cite[Chapter I, Lemma~4.1]{DB}, 
there exists
a  constant $\boldsymbol\gm>1$ depending only on the $\mathfrak{data}$, such that
if we require  
\begin{equation}\label{Eq:k-M:1}
Y_o=\frac{1}{k^p}\biint_{Q_{o}} (u-\ell)_+^p\,\dx\dt\leqslant  (1-\lm)^{N+p}\dl\lm^{N+sp}  ,
\end{equation}
then
$$
\lim_{n\to\infty}Y_n=0, \quad\text{i.e.}\quad u(x,t) \leqslant \ell(t) + k\quad\text{a.e. in}\>\> Q_\infty=Q(\lm\rho,\lm\theta).
$$ 
Taking both requirements \eqref{Eq:k-M} and \eqref{Eq:k-M:1}  of $k$ into account, we obtain that
\begin{align}\nonumber
\sup_{Q(\lm\rho,\lm \theta)} u &\leqslant \dl^{\frac{1}{p-1}}  \Big(\frac{R}{\rho}\Big)^{\frac{N}{p-1}} \sup_{Q(R,\theta)}u  + \Big(\frac{\rho^{sp}}{\theta}\Big)^\frac{1}{p-2}\\ \nonumber
&\quad+ \frac{1}{2}\Big(\frac{R}{R-\rho}\Big)^{N+sp}\int^0_{-\theta}\int_{\rn\setminus B_R}\frac{ u^{p-1}_{+}(y,\tau)}{|y|^{N+2s}}\,\dy\d\tau\\ \label{Eq:A:4}
&\quad+  \frac{\boldsymbol\gm }{(1-\lm)^{N+p}\dl\lm^{N+sp}}\bigg[\biint_{Q(\rho, \theta)} u_+^p\,\dx\dt\bigg]^{\frac1p}.
\end{align}

Our next goal is to refine \eqref{Eq:A:4}, which holds true for all $Q(R,\theta)\subset E_T$, $\dl,\,\lm\in(0,1)$ and $\rho\in(0,R)$.
We particularly aim to absorb the first term on the right-hand side of \eqref{Eq:A:4} into the left, and  to reduce the integral exponent in the last term  to $p-1$.  To this end,  we first estimate the last term in \eqref{Eq:A:4}  by Young's inequality:
\begin{align*}
&\frac{\boldsymbol\gm }{(1-\lm)^{N+p}\dl\lm^{N+sp}}\bigg[\biint_{Q(\rho, \theta)} u_+^p\,\dx\dt\bigg]^{\frac1p}\\
&\quad\leqslant\frac{\boldsymbol\gm }{(1-\lm)^{N+p}\dl\lm^{N+sp}}\bigg[\biint_{Q(\rho, \theta)} u_+^{p-1}\,\dx\dt\bigg]^{\frac1p}\Big[\sup_{Q(R,\theta)} u\Big]^{\frac{1}{p}}\\
&\quad\leqslant \dl^{\frac{1}{p-1}} \sup_{Q(R,\theta)} u + \frac{\boldsymbol\gm }{[(1-\lm)^{N+p}\dl\lm^{N+sp}]^{\frac{p}{p-1}}\dl^{\frac{1}{(p-1)^2}}} \bigg[ \biint_{Q(\rho, \theta)} u_+^{p-1}\,\dx\dt\bigg]^{\frac1{p-1}}.
\end{align*}
We plug this estimate back to \eqref{Eq:A:4}   to get
\begin{align}\nonumber
\sup_{Q(\lm\rho,\lm \theta)} u &\leqslant \dl^{\frac{1}{p-1}}\bigg[1+\Big(\frac{R}{\rho}\Big)^{\frac{N}{p-1}}\bigg] \sup_{Q(R,\theta)} u +\Big(\frac{\rho^{sp}}{\theta}\Big)^\frac{1}{p-2}\\ \nonumber
&\quad+ \frac12\Big(\frac{R}{R-\rho}\Big)^{N+sp} {\rm Tail} [u_+; Q(R, \theta)]\\ \label{Eq:A:5}
&\quad+  \frac{\boldsymbol\gm }{[(1-\lm)^{N+p}\dl\lm^{N+sp}]^{\frac{p}{p-1}}\dl^{\frac{1}{(p-1)^2}}} \bigg[ \biint_{Q(\rho, \theta)} u_+^{p-1}\,\dx\dt\bigg]^{\frac1{p-1}},
\end{align}
 for all $Q(R,\theta)\subset E_T$, $\dl,\,\lm\in(0,1)$ and $\rho\in(0,R)$.

Next, we use \eqref{Eq:A:5} to perform an interpolation argument. To this end, introduce for  $n\in\nn_0$
\begin{equation*}
\left\{
\begin{array}{cc}
\dsty\rho_n=  (1- 2^{-n-1})\rho,\quad\theta_n= (1 - 2^{-n-1}) \theta,\\ [5pt]
\dsty\widetilde{\rho}_n=\tfrac12(\rho_n +\rho_{n+1}),\quad \widetilde{\theta}_n=\tfrac12(\theta_n +\theta_{n+1}),\quad \lm_{n}=\frac{\rho_{n}}{\widetilde{\rho}_{n}}=\frac{\theta_{n}}{\widetilde{\theta}_{n}}.
\end{array}\right.
\end{equation*}
We apply \eqref{Eq:A:5} with $(\lm, \rho, R,  \theta)$ replaced by 
$(\lm_{n}, \widetilde{\rho}_n, \rho_{n+1}, \theta_{n+1})$ %Notice also that $\theta_n\leqslant\lm_n\theta_{n+1}$.
and obtain that for $n\in\nn_0$
\begin{align}\nonumber
\sup_{Q(\rho_n,\theta_n)} u &\leqslant \dl^{\frac{1}{p-1}}\bigg[1+\Big(\frac{\rho_{n+1}}{\widetilde{\rho}_n}\Big)^{\frac{N}{p-1}}\bigg] \sup_{Q(\rho_{n+1},\theta_{n+1})} u + \Big(\frac{\widetilde{\rho}_n^{sp}}{\theta_n}\Big)^\frac{1}{p-2}\\ \nonumber
&\quad +  \frac12\Big(\frac{\rho_{n+1}}{\rho_{n+1}-\widetilde{\rho}_n}\Big)^{N+sp} {\rm Tail} [u_+; Q(\rho_{n+1}, \theta_{n+1})]\\ \nonumber
&\quad+\frac{\boldsymbol\gm }{[(1-\lm_n)^{N+p}\lm_n^{N+sp}]^{\frac{p}{p-1}}\dl^{\frac{p(p-1)+1}{(p-1)^2}}} \bigg[ \biint_{Q(\widetilde{\rho}_n, \theta_{n+1})} u_+^{p-1}\,\dx\dt\bigg]^{\frac1{p-1}}\\ \label{Eq:A:6}
&=:\mathbf{I}_1+\mathbf{I}_2+\mathbf{I}_3+\mathbf{I}_4.
\end{align}
Let us continue to estimate the right-hand side of \eqref{Eq:A:6}. For the first term, we use the simple fact $\rho_{n+1}/\widetilde{\rho}_n\leqslant2$ to bound it, such that
\[
\mathbf{I}_1\leqslant \dl^{\frac{1}{p-1}} 4^{N}\sup_{Q(\rho_{n+1},\theta_{n+1})} u.
\] 
For the second term, we estimate
\[
\mathbf{I}_2=\Big(\frac{\widetilde{\rho}_n^{sp}}{\theta_n}\Big)^\frac{1}{p-2}\leqslant  \Big(\frac{2\rho^{sp}}{\theta}\Big)^\frac{1}{p-2}.
\]
The third term is estimated by
\begin{align*}
\mathbf{I}_3&=  \frac12\Big(\frac{2\rho_{n+1}}{\rho_{n+1}-\rho_n}\Big)^{N+sp} {\rm Tail} [u_+; Q(\rho_{n+1}, \theta_{n+1})] \\
 &\leqslant  2^{(n+3)(N+sp)}  {\rm Tail} [u_+; Q(\tfrac12 \rho, \theta)].
\end{align*}
To estimate the forth term, we use 
\begin{align*}
    \lm_n\in[\tfrac45, 1)\quad\text{and}\quad 1-\lm_n\geqslant 2^{-n-3}.
\end{align*}
Consequently, we get
\begin{align*}
\mathbf{I}_4\leqslant\frac{\boldsymbol\gm 2^{n\frac{p(N+p)}{p-1}} }{ \dl^{\frac{p(p-1)+1}{(p-1)^2}}} \bigg[ \biint_{Q(\rho, \theta)} u_+^{p-1}\,\dx\dt\bigg]^{\frac1{p-1}}.
\end{align*}
Combining all these estimates in \eqref{Eq:A:6}, we arrive at the recursive inequality for $n\in\nn_0$:
\begin{align*}
\sup_{Q(\rho_n,\theta_n)} u \leqslant \bar{\dl} \sup_{Q(\rho_{n+1},\theta_{n+1})} u +\mathcal{B}_{\dl} \boldsymbol{b}^n,
\end{align*}
where $\boldsymbol{b}:=2^{\frac{p(N+p)}{p-1}}$, $\bar{\dl}:=\dl^{\frac{1}{p-1}} 4^N$ and
\begin{align*}
\mathcal{B}_\dl:=  \Big(\frac{2\rho^{sp}}{\theta}\Big)^\frac{1}{p-2}+{\rm Tail} [u_+; Q(\tfrac12\rho, \theta)] +\frac{\boldsymbol\gm  }{ \dl^{\frac{p(p-1)+1}{(p-1)^2}}} \bigg[ \biint_{Q(\rho, \theta)} u_+^{p-1}\,\dx\dt\bigg]^{\frac1{p-1}}.
\end{align*}
Iterating this recursive inequality to obtain
\[
\sup_{Q(\rho_o,\theta_o)} u \leqslant \bar{\dl}^n \sup_{Q(\rho_{n},\theta_{n})} u +\mathcal{B}_{\dl} \sum_{i=0}^{n-1}(\bar{\dl}\boldsymbol{b})^i.
\]
Finally, we choose $\bar\dl=1/(2\boldsymbol{b})$ and let $n\to\infty$ to obtain
\begin{align*}
    \sup_{\frac12 Q(\rho,\theta)}u \leqslant \Big(\frac{2\rho^{sp}}{\theta}\Big)^\frac{1}{p-2}+  {\rm Tail} [u_+; Q(\tfrac12\rho, \theta)] + \boldsymbol\gm   \bigg[ \biint_{Q(\rho, \theta)} u_+^{p-1}\,\dx\dt\bigg]^{\frac1{p-1}}
\end{align*}
for some $\boldsymbol\gm = \boldsymbol\gm (\mathfrak{data})$.

Finally, let us apply the above estimate to the rescaled function $v$ that has been introduced in \eqref{Eq:u-to-v}. Then, 
\begin{align*}
    v(0,0)=1 \leqslant \Big(\frac{2 }{\theta}\Big)^\frac{1}{p-2}+  {\rm Tail} [v_+; Q(\tfrac12 , \theta)] + \boldsymbol\gm   \bigg[ \biint_{Q(1, \theta)} v_+^{p-1}\,\dx\dt\bigg]^{\frac1{p-1}}.
\end{align*}
We will control the first and the third terms.
Let us choose $\theta$ to satisfy
\begin{align*}
    \Big(\frac{2 }{\theta}\Big)^\frac{1}{p-2}\le\frac{1}{2} \quad\Longrightarrow\quad \theta= 2^{p-1}.
\end{align*}
Moreover, we estimate the last term by Young's inequality:
\begin{align*}
    \boldsymbol\gm   \bigg[ \biint_{Q(1, \theta)} v_+^{p-1}\,\dx\dt\bigg]^{\frac1{p-1}}
    \le \frac14+ 4^{p-2}\boldsymbol{\gm}^{p-1}\biint_{Q(1, \theta)} v_+^{p-1}\,\dx\dt.
\end{align*}
We combine these estimates  to obtain
\begin{align*}
    \frac{1}{4} &\leqslant   {\rm Tail} [v_+; Q(\tfrac12 , \theta)] + 4^{p-2}\boldsymbol{\gm}^{p-1}\biint_{Q(1, \theta)} v_+^{p-1}\,\dx\dt\\
    &\leqslant \boldsymbol{\gm}^{\prime\prime}_{\mathsf{h}} \int_{-\theta}^0\int_{\rn} \frac{v_+^{p-1}(x,t)}{(1+|x|)^{N+sp}}\,\dx\dt.
\end{align*}
The desired estimate can then be obtained by reverting to $u$. The final choice of $\boldsymbol{\gm}_{\mathsf{H}}$ is made out of the largest one among $4\boldsymbol{\gm}^{\prime\prime}_{\mathsf{h}}$ appearing in the above display, $\boldsymbol{\gm}^\prime_{\mathsf{h}}$  in \eqref{Eq:backward:0}, and $\boldsymbol{\gm}_{\mathsf{h}}$  in \eqref{proof:condi:2}
%%%%%%%%%%%%%%%%%%%%%%%%%%%%%%%%%%%%%%%%%%%%%%%%%%

%%%%%%%%%%%%%%%%%%%%%%%%%%%%%%%%%%%%%%%%%%%%%%%%%%

\section{Proof of the strong maximum principle}\label{S:Thm:2-proof}
We prove Theorem~\ref{Thm:2} in three steps. The argument presents some universal features that are independent of particular nonlinear problems. 

\underline{\it Step 1.}
Consider $(x_o,t_o)\in E\times(0,T)$ such that $u(x_o,t_o)>0$; in this step we aim to show that $u(x_o,t)>0$ for all $t\in(t_o,T)$. To this end, let us first  denote
\[
\mathsf{M}:=\sup_{[t_o,T]}u(x_o,\cdot)\quad\text{and}\quad \rho_*:=\min\Big\{\tfrac{1}{200}\dist(x_o,\pl E),\,1\Big\}.
\]
Then, let us choose 
$\rho_o $ to satisfy
\begin{align}\label{Eq:incl:0}\nonumber
&0<\rho_o\leqslant \rho_*, \>\>\text{and}\\
    &\big(t_o-2[\sig u(x_o,t_o)]^{2-p}(8\rho_o)^{sp}, t_o+ 2[\sig u(x_o,t_o)]^{2-p} (8\rho_o)^{sp} \big]\subset(0,T],
\end{align}
which plays the role of \eqref{Thm:condition:1}.
Assuming \eqref{Eq:incl:0}, the forward estimate in \eqref{Thm:conclusion:1} yields that
\begin{align*}
    \inf_{[t_o,t_1]}u(x_o,\cdot)\geqslant\sig u(x_o,t_o)\quad\text{where}\>\> t_1=t_o+[\sig u(x_o,t_o)]^{2-p}\rho_o^{sp}.
\end{align*}
Now, we specify the choice of $\rho_o$ in order to verify \eqref{Eq:incl:0}.
Indeed, the inclusion of time intervals in \eqref{Eq:incl:0} amounts to requiring that
\begin{equation}\label{Eq:time:0}
    2[\sig u(x_o,t_o)]^{2-p}(8\rho_o)^{sp}<\min\{t_o,\, T-t_o\}.
\end{equation}
Next, we choose $C>1$ such that
\begin{equation}\label{Eq:C-condition}
     \frac1C\min\{t_o,\, T-t_o\}\leqslant 2[\sig u(x_o,t_o)]^{2-p}(8\rho_*)^{sp};
\end{equation}
by $\min\{t_o,\, T-t_o\}\leqslant T$ and $u(x_o,t_o)\leqslant\mathsf{M}$, this is the case if $C$ satisfies
\[
    \frac{T}{C}\leqslant 2(\sig \mathsf{M})^{2-p}(8\rho_*)^{sp}
\]
which in turn indicates the choice 
\begin{equation}\label{Eq:C-choice}
C:=\max\bigg\{\frac{T}{2(\sig \mathsf{M})^{2-p}(8\rho_*)^{sp}},\, 2\bigg\}.
\end{equation}
Let us fix $C$ according to \eqref{Eq:C-choice} and choose $\rho_o$ via the equation
\[
2[\sig u(x_o,t_o)]^{2-p}(8\rho_o)^{sp}=\frac1C\min\{t_o,\, T-t_o\}.
\]
Such a choice conforms to \eqref{Eq:time:0}, which is the second condition in \eqref{Eq:incl:0}, and also fulfills the first condition in \eqref{Eq:incl:0} due to \eqref{Eq:C-condition}.

Now, suppose for some $i\in\N_0$ we have selected $\rho_i$ and $t_i$, such that
\begin{equation}\label{Eq:rho:i}
    2[\sig u(x_o,t_{i})]^{2-p}(8\rho_{i})^{sp}=\frac1C\min\{t_{i},\, T-t_{i}\},
\end{equation}
and  that
\begin{align}\label{Eq:incl:i}\nonumber
&0<\rho_i\leqslant \rho_* \>\>\text{and}\\
    &\big(t_i-2[\sig u(x_o,t_i)]^{2-p}(8\rho_i)^{sp}, t_i+ 2[\sig u(x_o,t_i)]^{2-p} (8\rho_i)^{sp} \big]\subset(0,T].
\end{align}
Moreover, we have obtained from the forward estimate in \eqref{Thm:conclusion:1} that
\begin{align}\label{Eq:forward:i}
    \inf_{[t_i,t_{i+1}]}u(x_o,\cdot)\geqslant\sig u(x_o,t_i)\quad\text{where}\>\> t_{i+1}=t_i+[\sig u(x_o,t_i)]^{2-p}\rho_i^{sp}.
\end{align}
Next, we aim to show the above estimate holds with  $i$ substituted by $i+1$, the crux being finding a proper $\rho_{i+1}$. For that, we basically repeat what has been done for $\rho_o$; we need to choose 
$\rho_{i+1} $ to satisfy
\begin{align}\label{Eq:incl:i+1}
&0<\rho_{i+1}\leqslant \rho_*  \>\>\text{and}\\ \nonumber
    &\big(t_{i+1}-2[\sig u(x_o,t_{i+1})]^{2-p}(8\rho_{i+1})^{sp}, t_{i+1}+ 2[\sig u(x_o,t_{i+1})]^{2-p} (8\rho_{i+1})^{sp} \big]%\\ \nonumber
    \subset(0,T].
\end{align}
Once $\rho_{i+1}$ is found from \eqref{Eq:incl:i+1}, we appeal to the forward estimate in \eqref{Thm:conclusion:1} to get that
\begin{align*}
    \inf_{[t_{i+1},t_{i+2}]}u(x_o,\cdot)\geqslant\sig u(x_o,t_{i+1})\quad\text{where}\>\> t_{i+2}=t_{i+1}+[\sig u(x_o,t_{i+1})]^{2-p}\rho_{i+1}^{sp}.
\end{align*}
Now, we specify the choice of $\rho_{i+1}$ in order to verify \eqref{Eq:incl:i+1}.
Indeed, the inclusion of time intervals in \eqref{Eq:incl:i+1} amounts to requiring that
\begin{equation}\label{Eq:time:i+1}
    2[\sig u(x_o,t_{i+1})]^{2-p}(8\rho_{i+1})^{sp}<\min\{t_{i+1},\, T-t_{i+1}\}.
\end{equation}
Next, we choose $C$ such that
\begin{equation}\label{Eq:C-condition-}
     \frac1C\min\{t_{i+1},\, T-t_{i+1}\}\leqslant 2[\sig u(x_o,t_{i+1})]^{2-p}(8\rho_*)^{sp};
\end{equation}
this is the case if $C$ is defined via \eqref{Eq:C-choice} because $\min\{t_{i+1},\, T-t_{i+1}\}\leqslant T$ and $u(x_o,t_{i+1})\leqslant \mathsf{M}$.
Let us fix $C$ according to \eqref{Eq:C-choice} and choose $\rho_{i+1}$ by
\[
2[\sig u(x_o,t_{i+1})]^{2-p}(8\rho_{i+1})^{sp}=\frac1C\min\{t_{i+1},\, T-t_{i+1}\}.
\]
Such a choice conforms to \eqref{Eq:time:i+1}, which is the second condition in \eqref{Eq:incl:i+1}, and also fulfills the first condition in \eqref{Eq:incl:i+1} due to \eqref{Eq:C-condition-}. Therefore, by induction we conclude that \eqref{Eq:incl:i} and \eqref{Eq:forward:i} hold  for all $i\in\N_0$.

Finally, let us compute from \eqref{Eq:rho:i} and \eqref{Eq:forward:i} that
\begin{equation}\label{Eq:t_i}
    t_{i+1}=t_i + \frac1{2C\cdot 8^{sp}}\min\{t_{i},\, T-t_{i}\}\qquad \forall \> i\in \N_0.
\end{equation}
The sequence defined via the recurrence \eqref{Eq:t_i} is non-decreasing and bounded by $T$. Then, it must have a limit, which we denote as $t_\infty\in \rr_+$. Let us send $i\to\infty$ in \eqref{Eq:t_i} and obtain that
\[
t_{\infty}=t_\infty + \frac1{2C\cdot 8^{sp}}\min\{t_{\infty},\, T-t_{\infty}\}
\]
from which we deduce that
\[
t_\infty=T.
\]
Meanwhile, note that the forward estimate \eqref{Eq:forward:i} for all $i\in\nn$ implies that
\begin{equation*}
    \inf_{[t_o,t_{i+1}]} u(x_o,\cdot)>0\qquad \forall \> i\in \N_0.
\end{equation*}
Therefore, we conclude the first step.

\underline{\it Step 2.} 
In this step, our goal is to show that if $E$ is connected and if $u(x_o,t_o)>0$ for some $(x_o,t_o)\in E\times(0,T)$, then $u(\cdot, t_o)>0$ on $E$. First observe that by continuity, the set $\{u(\cdot, t_o)>0\}\cap E$ is open in $E$. We will show that this set is also closed in $E$. Consequently, $\{u(\cdot, t_o)>0\}\cap E=E$ because $E$ is connected.

Let us pick a sequence $x_n\in \{u(\cdot, t_o)>0\}\cap E$, $n\in \nn$ and suppose they converge to some $x_o\in E$. It suffices to show $x_o\in \{u(\cdot, t_o)>0\}\cap E$.
Suppose to the contrary that $u(x_o,t_o)=0$ and define
\[
\rho_*:=\min\Big\{\tfrac{1}{200}\dist(x_o,\pl E),\,1\Big\}.
\]
Without loss of generality we may assume that 
\begin{equation}\label{Eq:x_n-x_o}
    |x_n-x_o|< \tfrac12\dist(x_o,\pl E)=100\rho_*\qquad \forall \> n\in\nn.
\end{equation}
Using this we estimate for an arbitrary $y\in\pl E$ that
\begin{align*}
   \dist(x_o,\pl E)&\leqslant |x_o-y| \leqslant |x_o-x_n| + |x_n-y|\\
  & \leqslant \tfrac12\dist(x_o,\pl E) + |x_n-y|,%\\& = \tfrac{1}{400}\dist(x_o,\pl E) + \dist(x_n,\pl E),
\end{align*}
from which we obtain
\begin{equation}\label{Eq:dist-compare}
    \dist(x_o,\pl E) \leqslant 2\dist(x_n,\pl E).
\end{equation}

Now, we select a sequence of radii $\rho_n$ such that
\begin{align}\label{Eq:incl:n}\nonumber
&0<\rho_n\leqslant \tfrac{1}{100}\dist(x_n,\pl E) \>\>\text{and}\\
    &\big(t_o-2[\sig u(x_n,t_o)]^{2-p}(8\rho_n)^{sp}, t_i+ 2[\sig u(x_n,t_o)]^{2-p} (8\rho_n)^{sp} \big]\subset(0,T],
\end{align}
which plays the role of \eqref{Thm:condition:1}.
The inclusion of time intervals in \eqref{Eq:incl:n} suggests we may choose $\rho_n$ via
\begin{equation}\label{Eq:choice-rho-n}
    2[\sig u(x_n,t_{o})]^{2-p}(8\rho_{n})^{sp}=\frac1C\min\{t_{o},\, T-t_{o}\}
\end{equation}
where $C>1$ (we slightly abuse the symbol) satisfies
\begin{equation}\label{Eq:C-condition+}
    \frac1C\min\{t_{o},\, T-t_{o}\}\leqslant 2[\sig u(x_n,t_{o})]^{2-p}(8\rho_{*})^{sp}.
\end{equation}
Using \eqref{Eq:C-condition+}  in \eqref{Eq:choice-rho-n}, we see that $\rho_n\leqslant \rho_*$ from which, recalling also the definition of $\rho_*$ and \eqref{Eq:dist-compare}, we deduce
\[
100\rho_n \leqslant \tfrac12\dist(x_o,\pl E) \leqslant  \dist(x_n,\pl E).
\]
This is exactly the first condition in \eqref{Eq:incl:n}. 

Now, let us make the final choice of $C$ from \eqref{Eq:C-condition+}, independent of $x_n$. By \eqref{Eq:x_n-x_o}, we have 
\begin{align*}
    \{x_n: n\in \nn\}\subset E_*:=\{x\in E: \dist(x,\pl E)\geqslant \rho_*\};
\end{align*}
since $E_*$ is compact, we denote (slightly abusing the symbol)
\[
\mathsf{M}:=\sup_{E_*} u(\cdot, t_o).
\]
Then, we will have \eqref{Eq:C-condition+} if we let
\begin{align*}
    C=\bigg\{\frac{T}{2(\sig \mathsf{M})^{2-p}(8\rho_*)^{sp}},\,  2\bigg\}.
\end{align*}
Such a choice of $C$ makes \eqref{Eq:incl:n} valid that takes the place of \eqref{Thm:condition:1}; the forward representation type estimate in \eqref{Thm:conclusion:2} applied with $\lm\rho_n$ for some $\lm\in(0,1)$ gives that
\begin{align}\label{Eq:apply-forward}
    \frac{1}{\boldsymbol{\gm}_{\mathsf{H}}}\,\int^{t_o-\kappa\theta_n(\lm\rho_n)^{sp}}_{t_o-\kappa\theta_n(2\lm\rho_n)^{sp}}\int_{\rn}\frac{u^{p-1}(x,t)}{ (\lm\rho_n+|x-x_n|)^{N+sp}}\,\dx\dt  \leqslant u(x_n,t_o) 
\end{align}
where we denoted $\theta_n=[u(x_n,t_o)]^{2-p}$. By the choice of $\rho_n$ in \eqref{Eq:choice-rho-n}, we compute that
\[
\kappa\theta_n(\lm\rho_n)^{sp}=\lm^{sp}\frac{\kappa\sig^{p-2}}{2\cdot 8^{sp}}\min\{t_{o},\, T-t_{o}\}=: \lm^{sp}\mathsf{b}.
\]
Hence, recalling also $\rho_n\leqslant \rho_*\leqslant 1$, we get from \eqref{Eq:apply-forward} that
\begin{align*}
        \frac{1}{\boldsymbol{\gm}_{\mathsf{H}}}\,\int^{t_o-\lm^{sp}\mathsf{b}}_{t_o-(2\lm)^{sp}\mathsf{b}}\int_{\rn}\frac{u^{p-1}(x,t)}{ (1+|x-x_n|)^{N+sp}}\,\dx\dt  \leqslant u(x_n,t_o) .
\end{align*}
Let us send $n\to\infty$ in the above display to obtain
\begin{align*}
        \frac{1}{\boldsymbol{\gm}_{\mathsf{H}}}\,\int^{t_o-\lm^{sp}\mathsf{b}}_{t_o-(2\lm)^{sp}\mathsf{b}}\int_{\rn}\frac{u^{p-1}(x,t)}{ (1+|x-x_o|)^{N+sp}}\,\dx\dt  \leqslant u(x_o,t_o) =0.
\end{align*}
Therefore, $u\equiv0$ a.e. in $\rn\times[t_o-(2\lm)^{sp}\mathsf{b},t_o- \lm^{sp}\mathsf{b}]$ for any $\lm\in(0,1)$, whence we send $\lm\to0$ and obtain that $u(\cdot, t_o)\equiv0$ on $E$ by continuity. This contradicts $u(x_n, t_o)>0$ and completes the proof.

\underline{\it Step 3.} 
Now, we combine what has been shown in the previous two steps to conclude that if $u(x_o,t_o)>0$, then $u>0$ in $E_{x_o}\times[t_o,T)$, where $E_{x_o}$ denotes the connected component of $E$ containing $x_o$. Equivalently, if $u(x_o,t_o)=0$, then $u\equiv0$ in $E_{x_o}\times(0,t_o]$. In addition, this conclusion implies that $u=0$ a.e. in $\rn\times(0,t_o]$. To see this, let us write down the weak formulation~\eqref{Eq:1:4p} in $E_{x_o}\times (t_1,t_o]$, $t_1>0$. Among the three terms on the left-hand side, the first two vanish due to $u\equiv0$ in $E_{x_o}\times(0,t_o]$, whence we obtain 
\begin{equation*}
\begin{aligned}
	\int_{t_1}^{t_o} \mathscr{E} (u(\cdot, t), \vp(\cdot, t) )\,\dt
	=0
\end{aligned}
\end{equation*}
for any
\begin{equation*}
\vp \in W^{1,2}_{\loc} (0,t_o;L^2(E_{x_o}) )\cap L^p_{\loc} (0,t_o;W_o^{s,p}(E_{x_o})),
\end{equation*}
where
\[
	 \mathscr{E}(t)=\int_{\rn}\int_{\rn}\mathsf{K}(x,y,t) \Phi_p (u(x,t) - u(y,t) ) (\vp(x,t) - \vp(y,t) )\,\dy\dx
\]
and $\rr\ni\xi\mapsto\Phi_p(\xi)=|\xi|^{p-2}\xi$.
Next, we use that $\vp(\cdot, t)$ is supported in $E_{x_o}$ and the symmetry of the kernel to split the integral over $\rn\times\rn$ into two terms. Namely,
\begin{align*}
    \int_{t_1}^{t_o} \mathscr{E} (u(\cdot, t), \vp(\cdot, t) )\, \dt=\int_{t_1}^{t_o} \mathscr{E}_1 (t)\, \dt+2\int_{t_1}^{t_o} \mathscr{E}_2(t)\, \dt,
\end{align*}
where
\[
\mathscr{E}_1(t):=\int_{E_{x_o}}\int_{E_{x_o}}\mathsf{K}(x,y,t) \Phi_p (u(x,t) - u(y,t) ) (\vp(x,t) - \vp(y,t) )\,\dy\dx
\]
and 
\[
\mathscr{E}_2(t):=\int_{E_{x_o}}\int_{\rn\setminus E_{x_o}}\mathsf{K}(x,y,t) \Phi_p (u(x,t) - u(y,t) ) (\vp(x,t) - \vp(y,t) )\,\dy\dx.
\]
The first term on the right-hand side vanishes thanks to $u\equiv0$ in $E_{x_o}\times(0,t_o]$, whereas the second term can be rewritten as
\begin{align*}
    2\int_{t_1}^{t_o}\int_{E_{x_o}}\int_{\rn\setminus E_{x_o}}\mathsf{K}(x,y,t) |0 - u(y,t) |^{p-2}  (0 - u(y,t) ) (\vp(x,t) - 0 )\,\dy\dx\dt
\end{align*}
as $\vp(\cdot, t)$ is supported in $E_{x_o}$. Hence, the weak formulation reduces to
\begin{align*}
    2\int_{t_1}^{t_o}\int_{E_{x_o}}\int_{\rn\setminus E_{x_o}}\mathsf{K}(x,y,t)  [u(y,t)] ^{p-1}     \vp(x,t)   \,\dy\dx\dt=0
\end{align*}
which gives by the lower bound of $\mathsf{K}$ in \eqref{Eq:K} that
\begin{align}\label{Eq:weak-form-reduced}
     \int_{t_1}^{t_o}\int_{E_{x_o}}\int_{\rn\setminus E_{x_o}}\frac{[u(y,t)] ^{p-1}     \vp(x,t)}{|x-y|^{N+sp}}     \,\dy\dx\dt=0
\end{align}
for any testing function $\vp$ satisfying \eqref{Eq:test-function} with $E_T$ replaced by $E_{x_o}\times(0,t_o]$. Let us suppose there is some $B_{2R}(x_o)$ inside $E_{x_o}$ and choose $\vp(x,t)=t\z(x)$ where $\z\ge0$ is smooth and supported in $B_R(x_o)$ but is not identical to zero. Then, observe that for $x\in B_R(x_o)$ and $y\in \rn\setminus E_{x_o}$ we have
\[
|x-y|\leqslant |x-x_o| + |x_o-y|\leqslant 2|x_o-y|.
\]
Hence, we get from \eqref{Eq:weak-form-reduced} that
\begin{align*}
    \int_{t_1}^{t_o}\int_{\rn\setminus E_{x_o}}\frac{[u(y,t)] ^{p-1}    t }{|x_o-y|^{N+sp}}     \,\dy\dt\cdot \int_{E_{x_o}}\z(x)\,\dx=0,
\end{align*}
whence we conclude the proof.

%%%%%%%%%%%%%%%%%%%%%%%%%%%%%%%%%%%%%%%%%%%%%%%%%%
%%%%%%%%%%%%%%%%%%%%%%%%%%%%%%%%%%%%%%%%%%%%%%%%%%
\section{Further applications}\label{S:appl}
As in \cite[\S~11]{DBGV-mono}, we point out that the scheme devised in our work is quite flexible and can be applied to other nonlinear integro-differential equations with measurable kernels as well. One example would be
\begin{equation}\label{Eq:pme}
\pl_t (|u|^{q-1}u) + \mathscr{L} u=0\quad\text{weakly in}\>\> E_T,
\end{equation}
where $q\in(0,1)$ and $\mathscr{L}$ is defined in \eqref{Eq:1:2}--\eqref{Eq:K} with $p=2$. Its notion of solutions can be formulated similar to Definition~\ref{Def:sol}. This kind of equations, in various setups, have been studied by a number of authors; see~\cite{Caff-10, Bonforte-17, Bonf-Fig-Vaz-18, Bonf-Vaz-14, vaz-12, maan-26}--just to mention a few. 

A continuity result paralleling that of \cite[Theorem~1.1]{Liao-mod-24} is important but still missing for \eqref{Eq:pme}, cf.~Remark~\ref{Rmk:continuity}. However, assuming qualitative continuity we have the following result.

\begin{theorem}[Harnack type estimates]%\label{Thm:pme}
   Let  $u$ be a continuous, local weak solution to \eqref{Eq:pme} with $q\in(0,1)$.
Assume that $u\geqslant0$ in $E_T$ and  $u_o:=u(x_o,t_o)>0$ for some $(x_o,t_o)\in E_T$. There exist constants $\boldsymbol{\gm}_{\mathsf{H}}>1$ and $\kappa,\,\sig\in(0,1)$ depending on $s$, $q$, $N$, $C_o$, and $C_1$, such that if  
\begin{equation*}%\label{Thm:condition:1}
    B_{100\rho}(x_o)\times (t_o-2\theta_{\sig}(8\rho)^{2s}, t_o+ 2\theta_{\sig} (8\rho)^{2s} ]\subset E_T,
\end{equation*} 
 where $\theta_{\sig}:=(\sig u_o)^{q-1}$ and if
\begin{equation*}%\label{Thm:condition:2}
    \boldsymbol{\gm}_{\mathsf{H}}\int_{t_o-2\theta_{\sig}(8\rho)^{2s}}^{t_o+2\theta_{\sig}(8\rho)^{2s}}\int_{\rn\setminus B_{\rho}(x_o)}\frac{u_-(x,t)}{|x-x_o|^{N+2s}}\,\dx\dt\leqslant\sig u_o,
\end{equation*}
then we have the following two conclusions.

\begin{itemize}
    \item[(i)] The local, intrinsic Harnack estimate holds true:
    \begin{equation*}%\label{Thm:conclusion:1}
        \sig\sup_{B_\rho(x_o)}u (\cdot, t_o-\kappa\theta_{1}\rho^{2s} ) \leqslant u_o \leqslant \sig^{-1}\inf_{B_\rho(x_o)}u (\cdot, t_o+\kappa\theta_{\sig}\rho^{2s} ).
    \end{equation*}
    \item[(ii)] The global,  intrinsic representation type estimate holds true:
    \begin{align*}%\label{Thm:conclusion:2}\nonumber
    \frac{\kappa\sig}{\boldsymbol{\gm}_{\mathsf{H}}}\,\int^{t_o-\kappa\theta_1\rho^{2s}}_{t_o-\kappa\theta_1(2\rho)^{2s}}\int_{\rn}&\frac{u_+(x,t)}{ (\rho+|x-x_o|)^{N+2s}}\,\dx\dt  \\
    &\leqslant u_o \leqslant  \boldsymbol{\gm}_{\mathsf{H}}\,\int^{t_o}_{t_o-2\theta_1\rho^{2s}}\int_{\rn}\frac{u_+(x,t)}{ (\rho+|x-x_o|)^{N+2s}}\,\dx\dt.
\end{align*}
\end{itemize}
\end{theorem}
Again, one should keep in mind that the results are entirely local, and no initial or boundary data have been prescribed.
By  a similar proof of Theorem~\ref{Thm:2}, we can also prove a strong maximum principle for \eqref{Eq:pme}.
Further generalizations are possible for a class of ``doubly nonlinear'' equations that encompass those considered here. Alternatively, one could consider general kernels, such as those in \cite{ChenJY, Kass-20}, and investigate the relationship among Harnack type estimates, Poincar\'e type inequalities and the doubling volume property in metric measure spaces. 

%%%%%%%%%%%%%%%%%%
%%%%%%%%%%%%%%%%%%
\

\noi{\bf Data Availability.} Data sharing not applicable to this article as no datasets were generated or analyzed during the current study.

\

\noi{\bf Conflict of Interest.} The author declares no conflict of interest.

\

\noi{\bf Use of AI.} No AI tools are used.
%%%%%%%%%%%%%%%%%%%%%%%%%%%%%%%%%%%%%%%
%%%%%%%%%%%%%%%%%%%%%%%%%%%%%%%%%%%%%%%
%%%%%%%%%%%%%%%%%%%%%%%%%%%%%%%%%%%%%%%
%%%%%%%%%%%%%%%%%%%%%%%%%%%%%%%%%%%%

\end{document}